\documentclass[preprint,sort]{elsarticle}

\usepackage{amsfonts,amsthm,amssymb,amsmath}
\usepackage[latin1]{inputenc}
\usepackage[T1]{fontenc}
\usepackage{ae,aecompl}
\usepackage{graphicx}
\usepackage{enumitem}
\usepackage{nameref,hyperref,cleveref}

\newtheorem{theo}{Theorem}[section]
\crefname{theo}{theorem}{theorems}
\Crefname{theo}{Theorem}{Theorems}
\newtheorem{prop}[theo]{Proposition}
\crefname{prop}{proposition}{propositions}
\Crefname{prop}{Proposition}{Propositions}
\newtheorem{lem}[theo]{Lemma}
\crefname{lem}{lemma}{lemmas}
\Crefname{lem}{Lemma}{Lemmas}
\newtheorem{cor}[theo]{Corollary}
\crefname{cor}{corollary}{corollaries}

\theoremstyle{definition}

\newtheorem{defi}[theo]{Definition}
\crefname{defi}{definition}{definitions}
\Crefname{defi}{Definition}{Definitions}
\newtheorem{con}[theo]{Conjecture}
\crefname{con}{conjecture}{conjectures}
\Crefname{con}{Conjecture}{Conjectures}

\theoremstyle{remark}

\newtheorem*{rem}{Remark}

\crefname{clm}{claim}{claims}
\Crefname{clm}{Claim}{Claims}

\theoremstyle{plain}

\crefformat{equation}{(#2#1#3)}
\Crefformat{equation}{(#2#1#3)}

\newcommand{\set}[1]{\left\lbrace #1 \right\rbrace}
\newcommand{\cont}[1]{\!\left\lbrace #1 \right\rbrace}
\newcommand{\card}[1]{\left\vert #1 \right\vert}

\DeclareMathOperator{\id}{id}

\journal{Journal of Combinatorial Theory, Series B}

\begin{document}

\begin{frontmatter}



\title{On spanning tree packings of highly edge connected graphs}


\author[fwf]{Florian Lehner}
\fntext[fwf]{The author acknowledges the support of the Austrian Science Fund (FWF), project W1230-N13.}

\address{TU Graz, Institut für Geometrie, Kopernikusgasse 24, 8010 Graz, Austria}

\begin{abstract}
We prove a refinement of the tree packing theorem by Tutte/Nash-Williams for finite graphs. This result is used to obtain a similar result for end faithful spanning tree packings in certain infinite graphs and consequently to establish a sufficient Hamiltonicity condition for the line graphs of such graphs.
\end{abstract}

\begin{keyword}
infinite graph theory \sep end faithful spanning tree \sep spanning tree packing \sep Hamiltonian cycle

\MSC[2010] 05C63 \sep 05C05

\end{keyword}

\end{frontmatter}

\section{Introduction}

A spanning tree packing of a graph is a set of edge disjoint spanning trees. The following theorem, discovered independently by Tutte and Nash-Williams, provides a sufficient condition for the existence of a spanning tree packing of cardinality $k$.
\begin{theo}[\citet{0096.38001}, \citet{0102.38805}]
\label{tuttenashwilliams}
Every finite $2k$-edge connected graph has $k$ edge disjoint spanning trees. 
\end{theo}

It is known that this result does not remain true for infinite graphs, not even for locally finite graphs, i.e., infinite graphs where every vertex has only finitely many neighbours. \citet{MR982868} showed that it is possible to construct locally finite graphs of arbitrarily high edge connectivity that do not possess two edge disjoint spanning trees.

Meanwhile the following approach due to \citet{1063.05076,1050.05071} allows a natural extension of the result to locally finite graphs. They proposed to use topological notions of paths and cycles in infinite graphs, which has the advantage of being able to define cycles containing infinitely many edges. More precisely they used homeomorphic images of the unit interval and the unit circle in the Freudenthal compactification of a graph (so called arcs and topological circles) as infinite analogues of paths and cycles in finite graphs. They also introduced the concept of a topological spanning tree, which is an infinite analogue of finite spanning trees compatible with the notions of arcs and topological circles.

Using the topological notions above numerous results from finite graph theory have been generalised  to locally finite graphs
\cite{
brewfunk,
1055.05088,
1077.05081,
pre05523086,
1088.05026,
1136.05030,
1180.05060,
1187.05031,
1067.05037,
1063.05076,
1050.05071,
agelos,
pre05498547,
pre05576421,
pre05686847,
1085.05052}, and even to general topological spaces \cite{graph-continua,Vella:2008fk},  substantiating their impact on infinite graph theory. As mentioned earlier these new concepts can also be used to establish a generalisation of \Cref{tuttenashwilliams}. The following result can be found in \cite{1086.05001}.

\begin{theo}
\label{toptrees}
Every locally finite $2k$-edge connected graph has $k$ edge disjoint topological spanning trees. 
\end{theo}

The starting point of this paper was the following conjecture by Georgakopoulos related to the Hamiltonian problem in infinite graphs. 

\begin{con}[\citet{agelos}]
\label{agelos}
The line graph $L(G)$ of every locally finite $4$-edge connected graph $G$ has a Hamiltonian circle.
\end{con}

By an infinite Hamiltonian circle we mean a topological circle containing every vertex and every end. In the finite case \Cref{agelos} is known to be true. This was first observed by \citet{MR856118} who stated the fact without proof. A simple proof later given by \citet{MR928734} makes use of \Cref{tuttenashwilliams}. However, it turns out that \Cref{toptrees} is insufficient to provide a similar proof for \Cref{agelos}. Hence we need a generalisation of the theorem of Tutte/Nash-Williams involving a better notion of spanning trees. 

In the present paper we establish a sufficient condition for the existence of large end faithful spanning tree packings similar to \Cref{tuttenashwilliams} where an end faithful spanning tree packing is a spanning tree packing in which every tree is end faithful.

\begin{theo}
\label{onetreecountable}
Let $G$ be a $2k$-edge connected locally finite graph with at most countably many ends. Then $G$ admits an end faithful spanning tree packing of cardinality $k-1$.
\end{theo}

We also show that for a graph with at most countably many ends the topological spanning tree packing in \Cref{toptrees} can be chosen in a way that the union of any two of the topological spanning trees is an end faithful subgraph of $G$.

\begin{theo}
\label{twotreescountable}
Let $G$ be a $2k$-edge connected locally finite graph with at most countably many ends. Then $G$ admits a topological spanning tree packing $\mathcal T$ of cardinality $k$ such that $T_1 \cup T_2$ is an end faithful connected spanning subgraph of $G$ whenever $T_1, T_2 \in \mathcal T$ and $T_1 \neq T_2$.
\end{theo}

Note that the counterexample of \citet{MR982868} mentioned earlier shows that there is a locally finite graph with uncountably many ends which does not permit two edge disjoint connected spanning subgraphs. Hence \Cref{onetreecountable,twotreescountable} are optimal in a sense that they do not hold if we consider graphs with uncountably many ends. Still, it may be possible to improve on the edge connectivity, that is, $2k$-edge connectivity may be sufficient for the existence of an end faithful spanning tree packing of cardinality $k$.

In the proof of the two results we use the following non-trivial refinement of \Cref{tuttenashwilliams} for finite graphs which we prove in \Cref{finite}.

\begin{theo}
\label{treecut2}
Let $G = (V,E)$ be a finite $2k$-edge connected graph and let $u,v \in V$ be such that $E_u := \set{e \in E \mid e \text{ is incident to }u}$ is a $u$-$v$-cut of minimal cardinality. Let $a$ and $b$ be vertices in the same component of $G - \set{u,v}$. Then there is a spanning tree packing $\mathcal T$ of cardinality $k$ of $G - u$ and two trees $T, T' \in \mathcal T$ such that $a$ and $b$ lie in the same component of $(T \cup T') - v$. 
\end{theo}

Using the tree packing results above we establish a sufficient Hamiltonicity condition for line graphs of locally finite graphs. This extends a recent result by \citet{brewfunk} and partially verifies \Cref{agelos}.

\begin{theo}
\label{linegraph}
The line graph of every locally finite $6$-edge connected graph with at most countably many ends has a Hamiltonian circle.
\end{theo}

Finally we outline the example of \citet{MR982868} and show that it is no counterexample to \Cref{agelos}.

\section{Basic Definitions and Facts}
\label{definitions}

All graphs considered in this paper are multigraphs, i.e., we allow multiple edges but no loops. For any notions that are not explicitly defined we will be using the terminology of \citet{1086.05001}. 

For the remainder of this section let $G=(V,E)$ be a connected (multi)graph.

A set $S \subseteq E$ is called an \emph{edge cut} or simply a \emph{cut}, if $G - S$ is not connected. A cut $S$ \emph{separates} two sets of vertices $U$ and $U'$ if there is no $U$-$U'$-path in $G - S$. In this case $S$ is called a \emph{$U$-$U'$-cut}. For convenience we omit the brackets for a $\set u$-$\set {u'}$-cut.

For sets of vertices $U, U' \subseteq V$ define the \emph{local edge connectivity} $\kappa'_G(U,U')$ as the minimal cardinality of a $U$-$U'$-edge cut. We write $\kappa'(U,U')$ instead of $\kappa'_G(U,U')$ if $G$ is clear from the context. The \emph{(global) edge connectivity} of $G$ is defined as $\kappa'(G) = \min_{U,U' \subseteq V} \kappa'_G (U,U').$ Note that $\kappa' (G)$ is the minimal number of edges we have to remove to disconnect the graph. 

The following well known theorem by Menger provides an alternative characterisation of local and thus also global edge connectivity. We will sometimes use this characterisation without explicitly mentioning it.
\begin{theo}[\citet{53.0561.01}]
\label{menger}
Let $G = (V,E)$ be a finite graph and $U, U' \subseteq V$. Then $\kappa' (U,U') \geq k$ if and only if there are $k$ edge disjoint $U$-$U'$-paths.
\end{theo}

If $U \subseteq V$ then the \emph{contraction} of $U$ in $G$ is defined by $G/U = (V', E')$ where $V'$ is obtained from $V$ by replacing the set $U$ by a new vertex $x_U \notin V$ and $E'$ is obtained from $E$ by replacing all endpoints of edges in $U$ by $x_U$ and deleting all loops. We will also call a graph a contraction of $G$  if it can be obtained from $G$ by a sequence of contractions. If $A,B \subseteq V$ such that at most one of the sets has non-empty intersection with $U$ then clearly an $A$-$B$-cut of minimal cardinality in $G$ will be taken to an $A'$-$B'$-cut of minimal cardinality in $G/U$ where $A'$ and $B'$ are the subsets of $V'$ corresponding to $A$ and $B$ respectively. In particular contractions do not decrease edge connectivity.

The \emph{subgraph induced} by $U$ in $G$ is defined as $G[U] = (U, E \cap U^2)$.

The \emph{minor induced} by $U$ in $G$ is the graph obtained from $G$ by contracting every component of $G - U$ to a vertex and denoted by  $G\cont U$.

Let $H =(V', E')$ be a subgraph of $G$ and let $G'$ be a contraction of $G$. The \emph{restriction} of $H$ to $G'$ is defined as the subgraph of $G'$ that contains exactly the edges corresponding to edges of $H$ and is denoted by $H \vert _{G'}$.

Let $s$ be a vertex of degree $\geq 2$. If $u$ and $v$ are neighbours of $s$, \emph{splitting off} the pair of edges $\set{us,vs}$ means deleting these two edges and replacing them by a new edge $uv$.

The inverse operation of splitting off is called \emph{pinching}, i.e., pinching a set of edges $E' \subseteq E$ at a vertex $w$ means replacing every edge $uv \in E'$  by the two edges $uw$ and $vw$. Note that if $w \notin V$ we need to add $w$ to $V$ first. For convenience we will omit the brackets if $E'=\set e$.

The following theorem by Mader will be an important ingredient to the proofs of our main results. A bridge here means a cut of cardinality one, i.e., an edge whose removal disconnects the graph (note that such an edge cannot exist if $\kappa'(G) > 1$).

\begin{theo}[\citet{0389.05042}]
\label{mader}
Let $G = (V,E)$ be a finite connected graph, $s \in V$ not incident to any bridges in $G$ and $\deg (s) \neq 3$. Then we can find a pair of  edges incident to $s$ that can be split off such that the local edge connectivity remains unchanged for all $x,y \in V \setminus \set{s}$.
\end{theo}

A \emph{ray} is a one-sided infinite path, the subrays of a ray is called its \emph{tails}. We call two rays \emph{equivalent} if there is no finite cut $S$ such that tails of the two rays lie in different components of $G-S$. The equivalence classes with respect to this relation are called the \emph{ends} of $G$, the set of ends is denoted by $\Omega(G)$. An \emph{$\omega$-ray} is a ray contained in the equivalence class $\omega \in \Omega$. While technically rays are paths, the word path will mean a finite path unless explicitly stated otherwise.

Endow $G$ with the topology of a 1-complex, i.e., every edge is homeomorphic to the real unit interval with any two of them being internally disjoint. Edges that share a vertex are glued together at one endpoint. We can extend this topology to $G \cup \Omega(G)$ by defining basic open neighbourhoods of $\omega \in \Omega$ in the following way:  let $S$ be a finite cut and denote by $C_S$ the component of $G - S$ in which all $\omega$-rays lie. Let $\Omega_S$ be the set of ends containing a ray in $C_S$ and denote by $E_S$ a set containing a half open edge of every edge connecting $C_S$ to $G-C_S$. Then $C_S \cup \Omega_S \cup E_S$ is a basic open neighbourhood of $\omega$. The arising topological space is called the \emph{Freudenthal compactification} or \emph{end compactification} of $G$ and denoted by $\overline G$. The set $\Omega$ of ends of $G$ with the induced topology is called the \emph{end space} of $G$.

There are some ambiguous notions appearing in both topology and graph theory, like connectedness, which may not coincide for a subgraph $H$ of $G$ and its closure in $\overline G$. Throughout this paper these notions will refer to graph theoretical concepts. If we use the topological notions we will state it explicitly.

An \emph{infinite star} is the complete bipartite graph $K_{1,\vert \mathbb N \vert}$. A \emph{subdivision} of a graph is obtained from this graph by replacing every edge by a finite path. An \emph{infinite comb} is a graph consisting of a ray $\gamma$ and infinitely many disjoint paths having exactly their first vertex on $\gamma$. The other endpoints of the paths are called the \emph{teeth} of the comb, $\gamma$ is called its \emph{spine}. Note that a tooth may lie on the spine if the respective path has length zero. The following lemma is a standard result in infinite graph theory, for a proof see \citet{1086.05001}.

\begin{lem}[Star-Comb-Lemma]
\label{starcomb}
Let $U \subseteq V$ be an infinite set of vertices of $G$. Then $G$ contains either a subdivision of an infinite star with all leaves in $U$ or an infinite comb with all teeth in $U$.
\end{lem}
Clearly in a locally finite graph always the latter is the case. Another easily observed fact is that the sequence of the teeth of a comb converges to the end in which its spine lies. 

Let $H$ be a subgraph of $G$. Then there is a unique continuous function $\varphi \colon \overline H \to \overline G$ such that $\varphi \mid _H = \id$. We say that $H$ is \emph{end faithful} if $\varphi \mid_{\Omega(H)}$ is a homeomorphism of the end spaces of $H$ and $G$. Equivalently, $H$ is an end faithful subgraph if it contains an $\omega$-ray for every end $\omega$ of $G$ and any two rays in $H$ which lie in the same end of $G$ also lie in the same end of $H$. Note that a spanning tree $T$ of $G$ is end faithful if and only if for every end $\omega$ of $G$ any two $\omega$-rays in $T$ have a common tail.

A \emph{topological path} is a continuous (but not necessarily injective) map from the closed unit interval $[0,1]$ to $\overline G$. The image of an injective topological path is called an \emph{arc}.

We call a homeomorphic image of the half open unit interval $[0,1)$ in $\overline G$ a \emph{topological ray}. Analogous to ordinary rays we call the image of $[a,1)$ for $a \in [0,1)$ a \emph{topological tail} of that ray. We say that a topological ray $\gamma$ \emph{converges to $x \in \overline G$} if defining $\gamma(1) =x$ would make it a topological path. Notice that all of its topological tails converge to $x$ as well. Usually $x$ will be an end of $G$.

A homeomorphic image of the unit circle $C^1$ in $\overline G$ is called a \emph{topological circle}.

A topologically path-connected subspace of $\overline G$ that does not contain a topological circle is called a \emph{topological tree}. A \emph{topological spanning tree} of $G$ is a topological tree that contains all vertices and all ends of $G$ and every edge of which it contains an inner point.

A \emph{(topological) spanning tree packing} of $G$ is a set of pairwise edge disjoint (topological) spanning trees of $G$.

For the sake of simplicity we will not distinguish between a subgraph of $G$ and its closure in $\overline G$, for example, we will call a subgraph $T$ of $G$ a topological spanning tree of $G$ if its closure is a topological spanning tree in $\overline G$.

\section{Spanning Trees and Topological Spanning Trees}
\label{trees}

The following section contains a collection of facts about spanning trees and topological spanning trees that we will use throughout this paper.

\begin{prop}
\label{contract2}
Let $G = (V,E)$ be a locally finite graph, $U \subseteq V$ and let $\set{U_i \mid i \in I}$ be a partitioning of $V \setminus U$ such that $G[U_i]$ is connected and there are only finitely many edges connecting $U_i$ to $V\setminus U_i$ for every $i$. Denote by $G_U$ the graph obtained from $G$ by contracting every $U_i$. Given spanning trees $T_U$ of $G_U$ and $T_i$ of $G[U_i]$ there is a spanning tree $T$ of $G$ satisfying $T \vert _{G_U} = T_U$ and $T[U_i] = T_i$.

The statement still holds if we replace spanning trees by topological spanning trees.
\end{prop}

The proof is easy and straightforward and will be left to the reader. Note that since $G$ is locally finite $I$ is at most countable. If we start with sets of $k$ edge disjoint spanning trees of each of the $G[U_i]$ and of $G_U$ then by \Cref{contract2} we can find $k$ edge disjoint spanning trees of $G$.

\begin{defi}
Let $G = (V,E)$ be a graph. A non-decreasing sequence $(V_n)_{n \in \mathbb N}$ of subsets of $V$ is called \emph{exhausting} if $\lim_{n \to \infty} V_n = V$.
\end{defi}

\begin{lem}
\label{limtree}
Let $G = (V,E)$ be a locally finite graph and let $V_n$ be an exhausting sequence of subsets of $V$. For every $n \in \mathbb N$ let $\set{U_i^n \mid i \in I}$ be a partitioning of $V \setminus V_n$ such that for every $i$ the graph $G[U_i^n]$ is connected and there are only finitely many edges connecting $U_i$ to $V\setminus U_i$. Let $G_n$ be the graph obtained from $G$ by contracting every $U_i^n$ and assume that for every $U_i^n$ there is a $U_j^{n-1}$ such that $U_i^n \subseteq U_j^{n-1}$, i.e., $G_{n-1}$ can also be seen as a contraction of $G_{n}$. If $T_n$ is a topological spanning tree of $G_n$ and $T_{n+1}\vert_{G_{n}}= T_{n}$ for every $n \in \mathbb N$ then $T := \lim_{n \rightarrow \infty} T_n \cup \Omega(G)$ is a topological spanning tree of $G$.
\end{lem}

\begin{proof}
The limit exists because the edge sets of $T_n$ form a non-decreasing sequence of subsets of $E$ and $V_n$ is exhausting.

We claim that there is a topological $u$-$v$-path $P \subseteq T$ for every pair $u,v \in V$. \citet{pre05576421} showed that, given a sequence $\tau _n$ of topological $u$-$v$-paths there is a subsequence $\tau_{n_i}$such that $\liminf \tau_{n_i}$ contains a topological $u$-$v$-path. If we choose each $\tau_n$ in a way that its intersection with $G[V_n]$ is contained in $T_n$ then $\liminf \tau_{n_i}$ is contained in $T$.

If one or both of $u$ and $v$ are ends a similar argument works. Just observe that Georgakopoulos' result also holds for ends. If $G_n$ contains a topological ray to an end then so does $T_n$ and thus we can choose the topological paths $\tau_n$ as required.

Finally assume that $T$ contains a topological circle $C$. Since a circle cannot consist entirely of ends $C$ has to contain at least one vertex, say $v$. Then $v$ is included in every $V_n$ from some index $n_0$ on. The restriction of $C$ to $T_n$ contains a topological circle for every $n \geq n_0$, a contradiction to $T_n$ being a tree.
\end{proof}

Note that unlike \Cref{contract2} the above result does not remain true if we substitute spanning trees for topological spanning trees. The reason for this is that the limit of a sequence finite paths need not necessarily be a finite path. In fact it may happen that such a limit is not even connected.

The last result in this section has already been mentioned in the introduction. It provides a characterisation of the end faithful spanning trees of a graph.

\begin{theo}[\citet{1050.05071}]
\label{treetoptree}
If $G$ is locally finite, then a spanning tree of $G$ is end faithful if and only if its closure in $\overline G$ is a topological spanning tree of $G$.
\end{theo}

\section{Spanning Tree Packings in Finite Graphs}
\label{finite}

In this section we prove \Cref{treecut2} and derive several corollaries of this result. It turns out to be easier to prove a strengthening of this theorem, where $a$ and $b$ are allowed to be inner points of edges as well. In order to formulate this stronger version of \Cref{treecut2} we need to introduce some notation.

\begin{defi}
Let $G =(V,E)$ be a graph, $v \in V$ and let $a$ and $b$ be vertices or inner points of edges of $G$. An $a$-$b$-arc is called an \emph{$a$-$b$-bypass} of $v$ if it does not contain $v$.

The set of all spanning tree packings $\mathcal T$ of $G$ of cardinality $k$ such that there are $l$ trees in $\mathcal T$ whose union contains an $a$-$b$-bypass of $v$ will be denoted by $\mathfrak T^{k,l}_{G}(a,b,v)$.
\end{defi}

Using the above definition we can state the following lemma which clearly implies \Cref{treecut2}.

\begin{lem}
\label{treecut2:lemma}
Let $G = (V,E)$ be a finite $2k$-edge connected graph and let $u,v \in V$ be such that $E_u := \set{e \in E \mid e \text{ is incident to }u}$ is a $u$-$v$-cut of minimal cardinality. Let $a$ and $b$ be vertices or inner points of edges of $G - u$. If there is an $a$-$b$-bypass of $v$ in $G - u$ then
\begin{equation} \label{star}
\mathfrak T _{G - u}^{k,2}(a,b,v) \neq \emptyset.  \tag{\protect$*$}
\end{equation}
\end{lem}

We will prove the statement by induction on the number of vertices. Before doing so however, let us take a look at some consequences of this result. First of all let us restate \Cref{treecut2} which is easily seen to be a special case of \Cref{treecut2:lemma} where $a$ and $b$ are only allowed to be vertices.

\newtheorem*{restate:treecut2}{\Cref{treecut2}}
\begin{restate:treecut2}
Let $G = (V,E)$ be a finite $2k$-edge connected graph and let $u,v \in V$ be such that $E_u := \set{e \in E \mid e \text{ is incident to }u}$ is a $u$-$v$-cut of minimal cardinality. Let $a$ and $b$ be vertices in the same component of $G - \set{u,v}$. Then there is a spanning tree packing $\mathcal T$ of cardinality $k$ of $G - u$ and two trees $T, T' \in \mathcal T$ such that $a$ and $b$ lie in the same component of $(T \cup T') - v$. 
\end{restate:treecut2}

\begin{rem}
Note that in particular \Cref{treecut2} implies that $G - u$ has a spanning tree packing of cardinality $k$---simply let $a=b\neq v$ if $G$ has at least three vertices. If $G - u$ consists only of $v$ then the existence of such a spanning tree packing is trivial.
\end{rem}

The next corollary that we state will be used in the proof of \Cref{onetreecountable}.

\begin{cor}
\label{treecut3}
Let $G = (V,E)$ be a finite $2k$-edge connected graph, $k \geq 2$, and let $u,v \in V$ be such that $E_u$ is a $u$-$v$-cut of minimal cardinality. Let $a$ and $b$ be vertices of $G - u$. If there is an $a$-$b$-arc that does not contain $u$ and $v$ then $\mathfrak T _{G - u}^{k-1,1}(a,b,v) \neq \emptyset.$
\end{cor}

\begin{proof}
By \Cref{treecut2} we can find $\mathcal T = \set{T^1,T^2,\ldots, T^k} \in \mathfrak T _{G - u}^{k,2}(a,b,v)$.
We may without loss of generality assume that $T^1 \cup T^2$ contains an $a$-$b$-bypass of $v$. Since both $a$ and $b$ are vertices every $a$-$b$-arc is a simple graph theoretical path and thus cycle free. 

Every acyclic subgraph of $T^1 \cup T^2$ can be extended to a spanning tree $T'$ of $T^1 \cup T^2$ which is also a spanning tree of $G- u$. Clearly $T'$ contains the same $a$-$b$-bypass of $v$ as $T^1 \cup T^2$ and $T'$ and  $T^i$ are edge disjoint for $i>2$ because $T' \subseteq T^1 \cup T^2$. So $\mathcal T' := \set{T',T^3,\ldots, T^k} \in \mathfrak T _{G - u}^{k-1,1}(a,b,v)$.
\end{proof}

\begin{rem}
An analogous proof can be given if $a$ and $b$ are inner points of edges $e_a$ and $e_b$ as long as at least one of the edges is not incident to $v$. If both $e_a$ and $e_b$ are incident to $v$ then an $a$-$b$-bypass of $v$ induces a circle and hence such a bypass cannot be contained in a tree.
\end{rem}

The other corollaries in this section will not be used later in the paper, but are merely included as further examples of uses of \Cref{treecut2}

\begin{cor}
\label{treecut1}
Let $G = (V,E)$ be a finite $2k$-edge connected graph, let $u,v \in V$ and let $S$ be a $u$-$v$-cut of minimal cardinality. Then there is a spanning tree packing $\mathcal T$ of cardinality $k$ of $G$ such that every tree $T \in \mathcal T$ contains exactly one edge in $S$.
\end{cor}

\begin{proof}
Let $C_u$ and $C_v$ be the vertex sets of the two components of $G - S$ in which $u$ and $v$ lie respectively (since $S$ is a minimal cut these are the only components). 

Consider the graph $G/C_u$ and denote by $u'$ the vertex obtained from $C_u$ in $G/C_u$. Clearly $S = \set{e \in E(G/C_u) \mid u' \in e}$ is a $u'$-$v$-cut of minimal cardinality and we can apply \Cref{treecut2} in order to obtain $k$ edge disjoint spanning trees of $(G/C_u) -{u'} = G[C_v]$. An analogous argument yields $k$ edge disjoint spanning trees of $G[C_u]$.

By connecting a spanning tree of $G[C_v]$ and a spanning tree of $G[C_u]$ with an edge in $S$ we obtain a spanning tree of $G$ that uses exactly one edge of $S$. There are at least $2k$ edges in $S$ and we have $k$ edge disjoint spanning trees of $G[C_u]$ and $G[C_v]$ respectively. This allows to construct the desired spanning tree packing and completes the proof.
\end{proof}

The last corollary we would like to mention, which is sometimes attributed to Catlin \cite{0771.05059,1168.05039}, characterises the edge connectivity of a graph by the spanning tree packing number of certain subgraphs. It clearly constitutes a refinement of \Cref{tuttenashwilliams}.

\begin{cor}
\label{tuttenashwilliams2}
Let $G=(V,E)$ be a finite graph. 
\begin{enumerate}
\item $G$ is $2k$-edge connected if and only if $G - F$ admits a spanning tree packing of cardinality $k$ for every set $F$ of at most $k$ edges.
\item $G$ is $(2k+1)$-edge connected if and only if $G - F$ admits a spanning tree packing of cardinality $k$ for every set $F$ of at most $k+1$ edges.
\end{enumerate}
\end{cor}

\begin{proof}
\begin{enumerate}
\item
Necessity is easily seen. For if we remove $k$ edges we still get $k$ edge disjoint spanning trees. Thus we have to remove at least another $k$ edges in order to disconnect the graph. So the minimal cardinality of a cut has to be at least $2k$.

To show that the condition is sufficient let $F \subseteq E$ be a set of $k$ edges and denote by $G'$ the graph obtained from $G$ by pinching $F$ at a vertex $v \notin V$.

$G'$ is $2k$-edge connected because every cut contains at least $2k$ edges: 
\begin{itemize}[label=--]
\item For any two vertices $a,b \in V(G)$ there are $2k$ edge disjoint $a$-$b$-paths inherited from $G$. So the cardinality of any cut that separates $a$ and $b$ is at least $2k$. 
\item The only cut which does not disconnect two vertices in $V(G)$ is the cut $E_v$ which contains $2k$ edges.
\end{itemize}

It follows immediately that $E_v$ has to be a $v$-$x$-cut of minimal cardinality for an arbitrary vertex $x \in V$ because it is a cut of cardinality $2k$ in a $2k$-edge connected graph. Hence by \Cref{treecut2} we can find $k$ edge disjoint trees of $G'-v = G-F$.

\item 
The proof of necessity is analogous to the first part.

Now let $E'$ be a set of $k+1$ edges and $e' \in E'$. Then $G - e'$ is $2k$-edge connected and thus $\left( G - e' \right) - \left ( E' \setminus e' \right ) = G - E'$ has $k$ edge disjoint spanning trees  which proves sufficiency. \qedhere
\end{enumerate}
\end{proof}

In the remainder of this section we will prove \Cref{treecut2:lemma}.

\begin{proof}[Proof of \Cref{treecut2:lemma}]
As mentioned earlier we will prove this lemma by induction on the number of vertices. If $G$ is a graph on two vertices then $G - u$ consists only of $v$. So there cannot be two points $a$ and $b$ as claimed in the condition of \Cref{treecut2:lemma}. Hence induction starts at $\vert V \vert = 3$. 

Let $G$ be a $2k$-edge connected graph on three vertices $u$, $v$ and $w$ and let $E_u$ be a $u$-$v$-cut of minimal cardinality. 

A spanning tree of $G- u$ consists of a $vw$-edge. So in order to obtain a spanning tree packing  of cardinality $k$ we need to ensure that there are at least $k$ such edges. By Menger's \Cref{menger} there are as many edge disjoint $u$-$v$-paths as edges in $E_u$. In particular there is a $vw$-edge for every $uw$-edge and since $\deg (w) \geq 2k$ there are at least $k$ edges connecting $v$ and $w$.

Now $a$ and $b$ can be either inner points of $vw$-edges or equal to $w$. Either way, there are two $vw$-edges whose union contains an $a$-$b$-bypass of $v$. We can select these two edges to form trees of the spanning tree packing.

\medskip

For the induction step we may assume that $\vert E_u \vert > \kappa'(G)$ because \Cref{treecut2:lemma} holds for $G$ if and only if it holds for the graph obtained from $G$ by adding a bundle of parallel $uv$-edges. Just observe that those edges are irrelevant for the statement of the lemma.

Now let $S$ be a cut such that $\vert S \vert = \kappa'(G)$. The cut $S$ does not separate $u$ and $v$ since $E_u$ is a $u$-$v$-cut of minimal cardinality and $\vert E_u \vert > \kappa'(G)$. Denote by $C$ and $C'$ the vertex sets of the components of $G - S$ and assume without loss of generality that $u,v \in C$.

We will distinguish the following two cases in both of which we will show that \cref{star} holds:
\begin{enumerate}[leftmargin=*,label= Case \arabic*: , ref=case \arabic*]
\item \label{case1} $\card{C'}>1$.
\item \label{case2} $\card{C'}=1$, i.e., $C' = \set w$ for some $w \in V$.
\end{enumerate}
In \ref{case1} consider the graph $G/C'$. Denote by $x_{C'}$ the vertex in $G/C'$ that has been obtained by contracting the set $C'$ and define $a'=x_{C'}$ if $a$ lies in $G[C']$ and $a'=a$ otherwise. Analogously define $b'$ from $b$.

We now claim that
\begin{enumerate}[leftmargin=*,label=(\arabic*)]
\item \label{case1_1} $G[C']$ has $k$ edge disjoint spanning trees, 
\item \label{case1_2} $\mathfrak T_{(G/C') - u}^{k,2}(a',b',v) \neq \emptyset$, and
\item \label{case1_3} the statement \cref{star} follows from \ref{case1_1} and \ref{case1_2}.
\end{enumerate}

So let us first prove \ref{case1_1}. For this purpose denote by $x_C$ the vertex that corresponds to $C$ in $G/C$. It can easily be seen that $G[C'] = (G/C) - x_C$. Moreover $G/C$ is $2k$-edge connected and the set of edges incident to $x_C$ in $G/C$ is a cut of minimal cardinality in $G/C$. Since $\vert V \vert \geq 4$  and $u,v \in C$, $G/C$ has strictly less vertices than $G$ and we can by our induction hypothesis find $k$ edge disjoint spanning trees of $(G/C) -{x_C}$. This proves \ref{case1_1}.

Next we prove \ref{case1_2}. Clearly $G / C'$ has strictly less vertices than $G$, $G / C'$ is $2k$-edge connected and the set of edges incident to $u$ in $G / C'$ is a $u$-$v$-cut of minimal cardinality in this graph. The $a$-$b$-bypass of $v$ in $G - u$ corresponds to an $a'$-$b'$-bypass of $v$ in $(G / C') - u$. So $\mathfrak T_{(G / C') - u}^{k,2}(a',b',v) \neq \emptyset$ by the induction hypothesis. This completes the proof of \ref{case1_2}.

In order to prove \ref{case1_3} let $\mathcal T_{C'} = \set{T_{C'}^1,T_{C'}^2, \ldots, T_{C'}^k}$ be a spanning tree packing of $G[C']$ and let $\mathcal T_{G/C'} = \set{T_{G/C'}^1,T_{G/C'}^2, \ldots, T_{G/C'}^k} \in \mathfrak T^{k,2}_{(G/C') - u}(a',b',v)$. Since we can permute the trees in the packing freely we may without loss of generality assume that $T_{G/C'}^1 \cup T_{G/C'}^2$ contains an $a'$-$b'$-bypass of $v$. 

If $a$ lies in $G[C']$ we may assume that $a$ lies in $T_{C'}^1$ for the following reasons. If $a$ is a vertex it is clearly contained in every $T_{C'}^i$. If it is an inner point of an edge contained in some $T_{C'}^i$ we can swap $T_{C'}^1$ and $T_{C'}^i$. Finally, if none of the previous cases holds then $a$ is an inner point of an edge $e$ not contained in any of the $T_{C'}^i$. In this case we can modify $T_{C'}^1$ by adding $e$ and removing an arbitrary edge of the circle that has been closed by doing so.

By an analogous argument we may assume that if $b$ lies in $G[C']$ and if $b$ is not an inner point of an edge of $T_{C'}^1$ then $b$ lies in $T_{C'}^2$.

Now let $T^i$ be the subgraph of $G$ that is obtained by replacing $x_{C'}$ in $T_{G/C'}^i$ by $T_{C'}^i$. We claim that $\mathcal T = \set{T^i \mid 1 \leq i \leq k} \in \mathfrak T^{k,2}_{G}(a,b,v)$. By \Cref{contract2} $\mathcal T$ is a spanning tree packing. Thus it suffices to prove that $T^1 \cup T^2$ contains an $a$-$b$-bypass of $v$.
\begin{itemize}[label=--]
\item If both $a$ and $b$ lie in $G[C']$ let $x \in C'$ be an arbitrary vertex. Since $x$ is a vertex it is contained in every $T_{C'}^i$.  By our choice of $\mathcal T_{C'}$ it holds that $a,b \in T_{C'}^1 \cup T_{C'}^2$. Hence we can find an $a$-$x$-arc and an $x$-$b$-arc in $T_{C'}^1 \cup T_{C'}^2$. 

The union of these two arcs clearly contains an $a$-$b$-arc. This arc does not contain $v$ because $v \notin C'$, so it is an $a$-$b$-bypass of $v$ in $T^1 \cup T^2$.
\item If only one of $a$ and $b$, say $a$, is contained in $G[C']$ then $T_{G/C'}^1 \cup T_{G/C'}^2$ contains an $x_{C'}$-$b$-bypass of $v$. This implies that there is a vertex $x \in C'$ such that $T^1 \cup T^2$ contains an $x$-$b$-bypass $B$ of $v$.

For the same reason as before there is an $a$-$x$-arc $A$ in $T_{C'}^1 \cup T_{C'}^2$.

The union $A \cup B$ contains an $a$-$b$-arc which is an $a$-$b$-bypass of $v$ because $v$ is not contained in either of $A$ and $B$.
\item If both $a$ and $b$ are not contained in $G[C']$ there is an $a$-$b$-bypass of $v$ in $T_{G/C'}^1 \cup T_{G/C'}^2$.

If this bypass does not contain $x_{C'}$ then it is also an $a$-$b$-bypass of $v$ in $T^1 \cup T^2$ and we are done.

So assume that it does contain $x_{C'}$. In this case there are vertices $x_a, x_b \in G[V']$ such that $T^1 \cup T^2$ contains an $a$-$x_a$-bypass $A$ and an $x_b$-$b$-bypass $B$ of $v$. Furthermore there is an $x_a$-$x_b$-arc $A'$ in $T_{C'}^1$ since $T_{C'}^1$ is connected and both $x_a$ and $x_b$ are vertices and thus contained in $T_{C'}^1$.

Clearly $A \cup B \cup A'$ contains an $a$-$b$-bypass of $v$.
\end{itemize}
This completes the proof of \ref{case1_3} and thus \cref{star} holds in \ref{case1}.

\medskip

Now consider \ref{case2}, i.e., assume that $C' = \set w$. If $\deg (w) = \kappa'(G)$ is odd then $G$ is $(2k+1)$-edge connected. Hence removing an arbitrary edge incident to $w$ from $G$ will leave the graph $2k$-edge connected. We would like to apply the induction hypothesis later, so we want to select the edge that we delete in a way that $E_u$ remains a $u$-$v$-cut of minimal cardinality after deleting it.

To decide which edge to delete choose $\kappa'_G(u,v)$ edge disjoint $u$-$v$-paths. The total number of edges incident to $w$ used by these paths has to be even as a path that enters $w$ via one edge has to leave the vertex via another one. Thus there is at least one such edge, say $e$, which is not used by any of the paths. Hence $E_u$ will still be a $u$-$v$-cut of minimal cardinality after $e$ has been removed since there are still $\kappa'_G(u,v)$ edge disjoint $u$-$v$-paths.

Now assume that $\deg w$ is even (possibly after deleting an edge incident to $w$). By Mader's \Cref{mader} we can split off all edges incident to $w$ in pairs without changing the local edge connectivity of any pair of vertices in $V - w$. In particular the set of edges incident to $u$ remains a $u$-$v$-cut of minimal cardinality throughout this procedure. We then delete the---now isolated---vertex $w$ and denote the graph that we obtain by $H = (V_H,E_H)$. Now define a multiset $V'$ over $V_H$ (i.e., $V'$ consists of vertices but may contain the same vertex more than once) and $E' \subseteq E_H$ as follows:
\begin{itemize}[label=--]
\item Add a copy of $v' \in V_H$ to $V'$ for every $v'w$-edge that has been split off with a $uw$-edge. Add another copy of $v'$ if a $v'w$-edge has been deleted in order to make $\deg (w)$ even.
\item Add $e' \in E_H$ to $E'$ if it has been created by splitting off a pair of edges none of which is incident to $u$.
\end{itemize}

Note that $\vert V' \vert + \vert E' \vert \geq k$ because every vertex in $V'$ and every edge in $E'$ corresponds to at most two edges incident to $w$ and $\deg w \geq 2k$. We will need this fact later. It is also an easy observation that $G - u$ is obtained from $H - u$ by pinching $E'$ at $w$ and adding a $v'w$-edge for every $v' \in V'$. We now claim that
\begin{enumerate}[resume*]
\item \label{case2_1} whenever $\mathcal T_H = \set{T_H^1,T_H^2, \ldots , T_H^k}$ is a spanning tree packing of $H - u$ and $e_1,e_2$ are edges of $G - u$ not incident to $w$ that did not result from pinching an edge in $\bigcup _{i=3}^k T_H^i$ then there is a spanning tree packing $\mathcal T = \set{T^1,T^2,\ldots, T^k}$ of $G - u$ such that
\begin{enumerate}[label=(\alph*),ref=(\alph*)]
\item \label{case2_1_a} for every $e \in E \cap E_H$ it holds that $e \in T_H^i$ if and only if $e \in T^i$, and
\item \label{case2_1_b} $e_1$ and $e_2$ are contained in $T^1 \cup T^2$.
\end{enumerate} 
\end{enumerate}

The first step in the proof of \ref{case2_1} is to turn $\mathcal T_H$ into a tree packing $\mathcal T = \set{T^1,T^2, \ldots , T^k}$ of $G - u$ by the following pinching procedure, where a tree packing is a set of edge disjoint but not necessarily spanning trees in $G$.

Begin with $T^i = T_H^i$ for all $i$ then pinch one edge in $E'$ after the other at $w$ and modify the $T^i$ as follows. Let $e = xy$ be the next edge to be pinched.
\begin{enumerate}[leftmargin=*, widest=A, label=(\Alph{*}), ref=(\Alph{*})]
\item \label{a} If $e$ belongs to none of the $T_H^i$ we do not modify any $T^i$.
\item \label{b} If $e \in T_H^i$ and no edges incident to $w$ have been added to $T^i$ so far then we remove $e$ from $T^i$ and add both $xw$ and $yw$ to $T^i$ in order to replace $e$.
\item \label{c} If $e \in T_H^i$ and we have already added edges to $T^i$ before we also remove $e$ from $T^i$. In this case adding both $xw$ and $yw$ to $T^i$ would result in a circle containing $w$ because there is either a $x$-$w$-path or a $y$-$w$-path in $T^i$ that does not use $e$. Hence we only add the edge which is not contained in this circle to $T^i$. 
\end{enumerate}
Note that the $T^i$ remain trees in each of these steps and that no $e \in E \setminus E'$ is removed from or added to $T^i$. So after pinching all of $E'$ according to \ref{a}, \ref{b} and \ref{c} we obtain $k$ edge disjoint trees each of which spans $V_H$ and fulfils property \ref{case2_1_a}. 

Still some of the trees may not contain $w$ and thus not be spanning trees. Before dealing with this problem, however, we will take care of property \ref{case2_1_b}. 

The edge $e_1$ can only be contained in $T^i$ for some $i$ if it has been added to $T^i$ applying \ref{b} or \ref{c}. Since $e_1$ did not result from pinching an edge of $T^i$ for $i > 2$ this implies that either $e_1 \in T^1 \cup T^2$ (without loss of generality $e_1 \in T^1$) or $e_1$ is not contained in any of the $T^i$ at all. In the first case we do not modify any of the $T^i$. In the latter case, if $w \notin T^1$ adding $e_1$ to $T^1$ makes $T^1$ a spanning tree of $G- u$. If $w \in T^1$ then $T^1$ is a spanning tree of $G-u$ already. Adding $e_1$ to it completes a circle that contains $w$. Remove the other edge incident to $w$ in this circle from $T^1$ to obtain a tree again. If $e_2 \notin T^1 \cup T^2$ we can use the same procedure as above to obtain $e_2 \in T^2$. Since in this case only $T^2$ is modified $e_1$ remains in $T^1$.

In the above modifications we did not add or remove any edges in $E \cap E_H$ to any tree. So property \ref{case2_1_a} still holds.

Finally we need to ensure that all of the trees contain $w$. For this purpose define 
\begin{align*}
L &= \set{T^i \in \mathcal T \mid w \notin T^i}, \\
L' &= \set{e \in E_w \mid u \notin e \text{ and } \forall T^i \in \mathcal T: e \notin T^i},
\end{align*}
and let $l = \vert L \vert$ and $l' = \vert L' \vert$. All that is left to prove is that $l' - l \geq 0$ because in this case we can add $w$ to each tree in $L$ using an edge in $L'$.

Note that $l'-l$ is invariant under the above procedure for adding $e_1$ and $e_2$ to $T^1 \cup T^2$ because if $w \notin T^1$ both values decrease by one while in the case that $w \in T^1$ they are constant.

So it is sufficient to show that $l' - l \geq 0$ held before these modifications. At that time there is an edge incident to $w$ in $G - u$ which is not used by any tree for every vertex in $V'$ and for every edge in $E'$ for which \ref{b} is not applied. Since \ref{b} is applied exactly once per tree containing $w$ this implies that
\begin{align*}
l' &= \vert L' \vert \\
&= \left\vert V' \right\vert + \left( \left\vert E' \right\vert - \vert \mathcal T \setminus L \vert \right) \\
&= \left( \left\vert V' \right\vert + \left\vert E' \right\vert \right) - \left( \vert \mathcal T\vert - \vert L \vert \right) \\
&\geq k - (k - l) \\
&=l.
\end{align*}
This completes the proof of \ref{case2_1}. 

We will now distinguish the following subcases of \ref{case2} in each of which we will apply the induction hypothesis and \ref{case2_1} to show that \cref{star} holds. Note that we can apply the induction hypothesis to $H$ because $H$ has strictly less vertices than $G$ and the set $E_u$ of edges incident to $u$ is a $u$-$v$-cut of minimal cardinality in $H$.

\begin{enumerate}[leftmargin=*,label=Case 2\alph*:, ref=case 2\alph*]
\item \label{case21} both $a$ and $b$ lie in $H$ and there is an $a$-$b$-bypass of $v$ in $H - u$.
\item \label{case22} both $a$ and $b$ lie in $H$ and there is no $a$-$b$-bypass of $v$ in $H - u$.
\item \label{case23} $a$ lies in $H$ and $b$ is an inner point of an edge that has been split off to generate an edge $e' \in E'$.
\item \label{case24} $a$ lies in $H$ and $b$ is an inner point of a $v'w$-edge for $v' \in V'$.
\item \label{case25} both $a$ and $b$ are inner points of edges incident to $w$ or are equal to $w$.
\end{enumerate}

Clearly these cases exhaust all possibilities where $a$ and $b$ could lie. So all we need to show is that \cref{star} holds in each of them.

In \ref{case21} we can apply the induction hypothesis to find a spanning tree packing $\mathcal T_H = \set{T_H^1,T_H^2,\ldots,T_H^k} \in \mathfrak T_{H- u}^{k,2}(a,b,v)$. We may without loss of generality assume that there is an $a$-$b$-bypass $A$ of $v$ in $T_H^1 \cup T_H^2$. If $A$ contains no edge of $E'$ then it is also an $a$-$b$-bypass of $v$ in any spanning tree packing $\mathcal T$ obtained by \ref{case2_1}. 

So assume that there is at least one edge in $E'$ that is used by $A$. Denote by $e_a$ the first edge in $E'$ that we pass through when we traverse $A$ starting at $a$. Clearly $e_a$ is contained in $T_H^1 \cup T_H^2$. It is also easy to see that for one endpoint $a'$ of $e_a$ there is an $a$-$a'$-subarc $A_a$ of $A$ that does not contain any inner point of an edge in $E'$. Analogously define $e_b$ and find a $b$-$b'$-subarc $A_b$ of $A$ that does not contain any inner point of an edge in $E'$ for an endpoint $b'$ of $e_b$.

By \ref{case2_1} we can find a spanning tree packing $\mathcal T$ of $G- u$ such that there are two trees in $\mathcal T$ whose union contains $A_a$, $A_b$ and the edges $a'w$ and $b'w$. Clearly the union of the two paths and the two edges constitutes an $a$-$b$-bypass of $v$ completing the proof of \ref{case21}.

Next consider \ref{case22}. Since there is no $a$-$b$-bypass of $v$ in $H - u$ every such bypass in $G - u$ used $w$. Thus such a bypass in $G - u$ induces an $a$-$a'$-bypass and a $b$-$b'$-bypass of $v$ in $H - u$ where $a'$ and $b'$ are vertices in $V'$ or inner points of edges in $E'$.

We claim that in this case
\begin{enumerate}[leftmargin=*,label=(\arabic*)]\setcounter{enumi}{4}
\item \label{case2_2} there is a spanning tree packing $\mathcal T_H = \set{T_H^1,T_H^2,\ldots,T_H^k}$ of $H$ such that $T_H^1 \cup T_H^2$ contains an $a$-$a'$-bypass and a $b$-$b'$-bypass of $v$.
\end{enumerate}

By the induction hypothesis we can find a spanning tree packing $\mathcal T_a = \set{T_a^1,T_a^2,\ldots, T_a^k}$ of $H - u$ such that  $T_a^1 \cup T_a^2$ contains an $a$-$a'$-bypass of $v$. We can also find a spanning tree packing $\mathcal T_b = \set{T_b^1,T_b^2,\ldots, T_b^k}$ of $H - u$ such that  $T_b^1 \cup T_b^2$ contains a $b$-$b'$-bypass of $v$. 

Since there is no $a$-$b$-bypass of $v$ in $H - u$ we know that $v$ is a cut vertex of $H - u$ and that $a$ and $b$ lie in different components of $H - \set {u,v}$. Denote by $C_a$ the set of vertices of the component in which $a$ lies and let $C_b = V_H \setminus (C_a \cup \set {u,v})$. It is easy to see that $T_a^i \set{C_a}$---that is, the minor induced by $C_a$ in $T_a^i$---is a spanning tree of $(H - u) \set{C_a}$ and that $T_b^i \set{C_b}$ is a spanning tree of $(H - u) \set{C_b} = (H - u)[V_H \setminus C_a]$. The statement \ref{case2_2} follows from \Cref{contract2}.

Now let $\mathcal T_H$ be a spanning tree packing of $H$ as claimed in \ref{case2_2} and let $A$ be an $a$-$a'$-bypass of $v$ in $T_H^1 \cup T_H^2$. Define a vertex $a'' \in A$ and an $a$-$a''$-subarc $A'$ of $A$ as follows.

If $A$ contains an inner point of an edge in $E'$ let $e_a$ be the first edge in $E'$ that we pass through when we traverse $A$ starting at $a$. Let $a''$ be an endpoint of that edge such that $A$ contains an $a$-$a''$ subarc $A'$ that does not use inner points of any edge in $E'$. In this case an $a''w$-edge results from pinching $e_a \in T_H^1 \cup T_H^2$

If $A$ doesn't contain any inner point of an edge in $E'$ then $a' \in V'$. Let $a'' = a'$ and $A' = A$. In this case there is an $a''w$-edge which did not result from pinching any edge, in particular not from pinching an edge in $T^i$ for $i>2$.

Analogously define $b''$ and $B'$ from a $b$-$b'$-bypass $B$ of $v$.

Now we can by \ref{case2_1} find a spanning tree packing $\mathcal T = \set{T^1, T^2, \ldots , T^k}$ of $G - u$ such that $T^1 \cup T^2$ contains $A'$, $B'$ and the edges $a''w$ and $b''w$. Clearly the union of the two paths and the two edges constitutes an $a$-$b$-bypass of $w$ in $T^1 \cup T^2$. This completes the proof in \ref{case22}.

Next let us turn to \ref{case23}. In this case again the $a$-$b$-bypass of $v$ in $G - u$ becomes an $a$-$a'$-bypass in $H - u$ where $a'$ is a vertex in $V'$ or an inner point of an edge in $E'$. By the induction hypothesis we can find a spanning tree packing $\mathcal T_H = \set{T_H^1,T_H^2, \ldots , T_H^k}$ of $H - u$ such that $T_H^1 \cup T_H^2$ contains an $a$-$a'$-bypass $A$ of $v$.
Analogously to \ref{case22} define $a''$ and an $a$-$a''$-subarc $A'$ of $A$. 

Recall that in \ref{case23} the point $b$ is an inner point of an edge that has been split off to create an edge $e' \in E'$. We claim that

\begin{enumerate}[resume*]
\item \label{case2_3} we can choose $\mathcal T_H$ in a way that $e' \in T_H^1 \cup T_H^2$.
\end{enumerate}

If $a$ and $e'$ lie in the same component of $H - \set{u,v}$ we may assume that $a'$ has been an inner point of $e'$ in the first place and thus $e'\in T_H^1 \cup T_H^2$ holds.

So assume that $a$ and $e'$ lie in different components. In this case we can choose a vertex $c$ in the component in which $e'$ lies. Since there is a $c$-$b'$-bypass of $v$ in $H - u$ for every inner point $b'$ of $e'$ we can apply the induction hypothesis and \ref{case2_2} to find a spanning tree packing with the desired properties.

This proves \ref{case2_3} and thus we can assume that the edge $b''w$ of which $b$ is an inner point did not result from pinching an edge in $T_H^i$ for $i>2$. 

By \ref{case2_1} we can find a spanning tree packing $\mathcal T = \set{T^1, T^2, \ldots , T^k}$ of $G - u$ such that $T^1 \cup T^2$ contains $A'$ and the edges $a''w$ and $b''w$ whose union contains an $a$-$b$-bypass of $v$ which completes the proof of \ref{case23}.

In \ref{case24} once again apply the induction hypothesis to find a spanning tree packing $\mathcal T_H = \set{T_H^1,T_H^2, \ldots , T_H^k}$ of $H - u$ such that $T_H^1 \cup T_H^2$ contains an $a$-$a'$-bypass $A$ of $v$ where $a'$ is a vertex in $V'$ or an inner point of an edge in $E'$. As in the previous two cases define $a''$ and an $a$-$a''$-subarc $A'$ of $A$.

Now we can apply \ref{case2_1} to find a spanning tree packing $\mathcal T = \set{T^1, T^2, \ldots , T^k}$ of $G - u$ such that $T^1 \cup T^2$ contains $A'$ and the edges $a''w$ and $bw$ which proves \cref{star} in \ref{case24}.

In \ref{case25} apply the induction hypothesis to find $k$ edge disjoint spanning trees of $H - u$. If the edges $a'w$ and $b'w$ on which $a$ and $b$ lie were created by pinching some edges in $E'$ we may permute the trees such that these edges lie in $T^1 \cup T^2$ or in none of the trees at all. We then apply \ref{case2_1} to find a spanning tree packing $\mathcal T = \set{T^1, T^2, \ldots , T^k}$ of $G - u$ such that $T^1 \cup T^2$ contains both $a'w$ and $b'w$. This proves that \cref{star} holds in \ref{case25}. 

Since there are no more cases left it also completes the induction step and thus the proof of \Cref{treecut2:lemma}.
\end{proof}

\section{Results for Locally Finite Graphs}
\label{infinite}

In the following sections we will prove \Cref{onetreecountable,twotreescountable}. First let us recall the statements and give a brief outline of the proofs.

\newtheorem*{restate:onetreecountable}{\Cref{onetreecountable}}
\begin{restate:onetreecountable}
Let $G$ be a $2k$-edge connected locally finite graph with at most countably many ends. Then $G$ admits an end faithful spanning tree packing of cardinality $k-1$.
\end{restate:onetreecountable}

\newtheorem*{restate:twotreescountable}{\Cref{twotreescountable}}
\begin{restate:twotreescountable}
Let $G$ be a $2k$-edge connected locally finite graph with at most countably many ends Then $G$ admits a topological spanning tree packing $\mathcal T$ of cardinality $k$ such that $T_1 \cup T_2$ is an end faithful connected spanning subgraph of $G$ whenever $T_1, T_2 \in \mathcal T$ and $T_1 \neq T_2$.
\end{restate:twotreescountable}

The graphs in the condition of the two theorems may be infinite in two ways: they are infinite in depth, i.e., there are rays and thus also ends, and they are infinite in width, i.e., there are infinitely many ends. In the proofs we will tackle these two types of infiniteness one by one. First we show the statements for one ended graphs by decomposing them into (countably many) finite graphs. Afterwards---using the results from the first part---we decompose graphs with countably many ends into finite and one ended graphs. Now let us give a brief summary of the proof steps.

As mentioned above we first consider the case where $G$ has only one end. We choose a strictly increasing sequence $G_n$ of finite contractions converging to $G$. The decomposition mentioned above consists of the graphs $G_n - G_{n-1}$ which are clearly finite. We inductively construct a sequence of spanning tree packings $\mathcal T_n$ of $G_n$ by choosing a spanning tree packing of $G_n - G_{n-1}$ (here \Cref{treecut2} is essential) and combining this packing with the spanning tree packing $\mathcal T_{n-1}$.  Taking the limit of the sequence $\mathcal T_n$ and showing that it is an end faithful spanning tree packing of $G$ finishes the proof in the one ended case.

If $G$ has countably many ends we also select an increasing sequence $G_n$ of contractions converging to $G$. This time however all of the contractions are infinite and $G_{n+1}$ has exactly one end more than $G_n$, more precisely $G_{n+1}-G_n$ consists of a one ended and possibly a finite component.  We apply the first part of the proof to inductively construct a sequence $\mathcal T_n$ of end faithful spanning tree packings of $G_n$. Like in the one ended case the limit of the sequence $\mathcal T_n$ is the desired end faithful spanning tree packing.

Clearly the crux in both parts of the proof is finding conditions under which the sequences $\mathcal T_n$ converge to an end faithful spanning tree packing and constructing the $\mathcal T_n$ accordingly. So in each of the two proof steps we establish such a condition before the actual proof. 

In the second part of the proof it will also be important to choose the contractions in a way that we can apply the result of the first part. Hence before moving to the proofs of \Cref{onetreecountable,twotreescountable} we show a result concerning cuts which is then used to define the contractions.

\medskip

Now let us turn to the one ended case. We will prove the following statements for one ended graphs which are slightly stronger than \Cref{onetreecountable,twotreescountable}. The stronger assertions are of vital importance for using the results in the proofs for graphs with countably many ends.

\begin{theo}
\label{onetree}
Let $G = (V,E)$ be a $2k$-edge connected locally finite graph with only one end $\omega$ and let $V' \subseteq V$ such that the set of edges connecting $V'$ to $V \setminus V'$ form a $\omega$-$V'$-cut of minimal cardinality. Then $G- V'$ admits an end faithful spanning tree packing of cardinality $k-1$.
\end{theo}

\begin{theo}
\label{twotrees}
Let $G = (V,E)$ be a $2k$-edge connected locally finite graph with only one end $\omega$ and let $V' \subseteq V$ such that the set of edges connecting $V'$ to $V \setminus V'$ form a $\omega$-$V'$-cut of minimal cardinality. Then $G- V'$ admits a topological spanning tree packing $\mathcal T$ of cardinality $k$ such that $T_1 \cup T_2$ is an end faithful connected spanning subgraph of $G$ whenever $T_1, T_2 \in \mathcal T$ and $T_1 \neq T_2$.
\end{theo}

As mentioned earlier, before proving \Cref{twotrees,onetree}  we need a condition which ensures that the limit of a sequence of spanning trees of finite contractions is an end faithful spanning tree. Since the limit is a topological spanning tree by \Cref{limtree}---note that in the finite case the notions of spanning trees and topological spanning trees coincide---we only need to ensure that it is connected. 

Throughout the construction process we will take a look at potential pairs of components which will be formalised by the notion of gaps. 

In the proof of \Cref{onetree} we would like to avoid that such a pair of potential components can propagate throughout the construction process because in this case we get at least two components in the limit.

In \Cref{twotrees} the situation is slightly different. We allow the topological spanning trees to have more than one component. However, we require that the union with any other topological spanning tree in the limit is connected and end faithful. This will be achieved by bridging gaps, i.e., by constructing in each step paths in the union of two trees which connect the two potential components involved in a gap.

The following definition puts the intuitive notions above in a more formal context.

\begin{defi}
Let $G = (V,E)$ be a locally finite graph and $(V_n)_{n \in \mathbb N}$ an exhausting sequence of subsets of $V$. Define $G_n = G \cont {V_n}$, i.e., $G_n$ is the graph obtained from $G$ by contracting each component of $G - V_n$ to a single vertex.

\begin{itemize}[label=--]
\item Given a spanning tree $T_n$ of $G_n$ we call a pair of components of $T_n[V_n]$ a \emph{gap} of $T_n$ in $G_n$. If $(C_1, C_2)$ is a gap and $u  \in C_1$ and $v \in C_2$ we say that the gap \emph{separates} $u$ and $v$.

\item Assume that for every $k \in \mathbb N$ we have a spanning tree $T_k$ of $G_k$ and that $T_{k+1} \vert_{G_{k}} = T_{k}$. Let $m < n$ be natural numbers.

A gap $(C_1,C_2)$ of $T_n$ \emph{extends} a gap $(C_1', C_2')$ of $T_m$ if the vertex set of $C_i'$ is a subset of the vertex set of $C_i$ for $i \in \set{1,2}$. Note that a gap $(C_1', C_2')$ of $T_m$ cannot be extended by multiple gaps in $T_n$. Hence we will not distinguish between a gap and its extension. This allows to talk about gaps in the sequence $T_k$.

A gap $(C_1, C_2)$ of $T_m$ \emph{terminates} in $V_n$ if it is not extended by a gap of $T_n$. Note that in this case $C_1$ and $C_2$ both are contained in the same component of $T_n[V_n]$ and thus connected by a path in $T_n$ that does not use any contracted vertex. A gap that does not terminate is called a \emph{persistent gap} in the sequence $T_k$.

\item Given a gap $\Gamma = (C_1,C_2)$ of $T_n$ we say that a tree $T_n' \subseteq G_n$ \emph{bridges} the gap $\Gamma$ in $V_n$ if there is a path $P$ from $C_1$ to $C_2$ in $T_n \cup T_n'$ which does not use any contracted vertex. The set of edges in $P \cap T_n'$ is called a \emph{$\Gamma$-bridge} in $T_n'$.

A sequence $T_n'$ bridges a gap $\Gamma$ in a sequence $T_n$ infinitely often if there are arbitrarily large sets of disjoint $\Gamma$-bridges in $T_n'$ as $n \to \infty$. Note that this can only be the case if $\Gamma$ is a persistent gap.
\end{itemize}
\end{defi}

The following results verify the intuition that in order to obtain end faithful connected spanning subgraphs we only have to avoid persistent gaps. Note that they do not only hold for one-ended graphs but for arbitrary locally finite graphs, even those with uncountably many ends.
\begin{lem}
\label{bridges}
Let $G = (V,E)$ be a locally finite graph and $(V_n)_{n \in \mathbb N}$ an exhausting sequence of finite subsets of $V$.
Furthermore let $T_n, T_n'$ be sequences of spanning trees of $G \cont {V_n}$ such that $T_n \vert_{V_{n-1}} = T_{n-1}$ and $T_n' \vert_{V_{n-1}} = T_{n-1}'$. If the sequence $T_n'$ bridges every persistent gap of the sequence $T_n$ infinitely often then $H := T \cup T'$ is an end faithful connected spanning subgraph of $G$ where  $T = \lim _{n \rightarrow \infty} T_n$ and $T' = \lim _{n \rightarrow \infty} T_n'$.
\end{lem}

\begin{proof}
The graph $H$ is connected because any pair of components of $H$ would constitute a persistent gap in the sequence $T_n$ that is not bridged. 

So we only need to show that any two rays $\gamma_1$ and $\gamma_2$ in $H$ belonging to the same end $\omega$ of $G$ are equivalent in $H$. Assume that $\gamma_i$ belongs to an end $\omega_i$ of $H$ for $i \in \set{1,2}$ and that $\omega_1 \neq \omega_2$. 

Let $n$ be big enough that $\omega_1$ and $\omega_2$ lie in different components $C_H^1$ and $C_H^2$ of $H - V_n$. Note that $C_H^1$ and $C_H^2$ both are subsets of the same component $C_G$ of $G- V_n$ because $\gamma_1$ and $\gamma_2$ converge to the same end of $G$. For $i \in \set{1,2}$ denote by $v_i \in C_H^i$ a vertex of $\gamma_i$ such that all consecutive vertices lie in $C_H^i$ as well.

If $v_1$ and $v_2$ belonged to the same component of $T$ then the unique path in $T$ connecting $v_1$ and $v_2$ would use vertices in $V_n$. Hence this path would correspond to a cycle in $T_n$. So $v_1$ and $v_2$ belong to different components of $T$ and there is a persistent gap $\Gamma$ that separates $v_1$ and $v_2$.

We know that $\Gamma$ is bridged infinitely often, i.e., there are arbitrarily large sets $\mathcal P$ of paths in $H$ connecting $v_1$ and $v_2$ such that \[
\forall P_1,P_2 \in \mathcal P \colon P_1 \cap P_2 \cap E(T') = \emptyset.
\] 
If a path $P \in \mathcal P$ has non-empty intersection with $V_n$ it has to use at least one edge of $T'$ with one endpoint in $V_n$ because otherwise $P \mid_{G_n}$ would be a cycle in $T_n$. Hence the number of such paths is bounded by the number of edges with an endpoint in $V_n$. So if $\vert \mathcal P \vert$ is large enough then there is a path $P^* \in \mathcal P$ that connects $v_1$ and $v_2$ such that $P^*$ contains no vertex of $V_n$.

Consequently $v_1$ and $v_2$ lie in the same component of $H - V_n$, a contradiction.
\end{proof}

\begin{lem}
\label{noinfgaps}
Let $G = (V,E)$ be a locally finite graph and $(V_n)_{n \in \mathbb N}$ an exhausting sequence of finite subsets of $V$.
Furthermore let $T_n$ be a sequence of spanning trees of $G \cont {V_n}$ such that $T_n \vert_{V_{n-1}} = T_{n-1}$. If the sequence $T_n$ contains no persistent gaps then $T := \lim _{n \rightarrow \infty} T_n$
is an end faithful spanning tree of $G$.
\end{lem}

\begin{proof}
Since $T_n$ contains no persistent gaps it is clear that $T$ is end faithful and connected by \Cref{bridges} with $T_n' = T_n$. We also know that $T$ is a topological spanning tree and thus acyclic---note that finite circles constitute a special case of topological circles. Hence $T$ is both a spanning tree and a topological spanning tree and \Cref{treetoptree} completes the proof
\end{proof}

With these auxiliary results on hand we can finally move on to the proofs of \Cref{onetree,twotrees}. In both of the proofs we will construct a sequence of finite contractions and corresponding spanning tree packings fulfilling the condition of \Cref{bridges,noinfgaps} respectively. This will be done in detail for \Cref{twotrees} while for \Cref{onetree} we will only indicate the modifications to be made.

\begin{proof}[Proof of \Cref{twotrees}]

Denote by $\omega$ the unique end of $G$ and let $V_0 := V'$. It is easy to see that the following construction of an exhausting sequence of sets of vertices is always possible. 

For every $n \in \mathbb N$ choose a finite set $W_n \supseteq V_{n-1}$ of vertices containing all neighbours of $V_{n-1}$ such that $G[W_n \setminus V_{n-1}]$ is connected and that there is only one component of $G - W_n$. Let $S_n$ be a $W_n$-$\omega$-cut of minimal cardinality and define $V_n$ to be the component of $G - S_n$ in which $W_n$ lies. 

Next we construct a sequence $\mathcal T _n = \set{T_n^1,T_n^2, \ldots , T_n^k}$ of spanning tree packings of $G_n := G \cont{V_n}-V_0$ such that 
\begin{enumerate}[label=(\arabic*)]
\item \label{twotrees1} $T_{n+1}^i\mid_{G_n} = T_n^i$ and
\item \label{twotrees2} for $i \neq j$ every persistent gap in $T_n^i$ is bridged infinitely often by $T_n^j$.
\end{enumerate}
Then $T^i = \lim_{n \to \infty} T_n^i$ is a topological spanning tree of $G - V'$ by \Cref{limtree}. Whenever $i \neq j$ the edge sets of $T^i$ and $T^j$ are disjoint because $T_n^i$ and $T_n^j$ are edge disjoint for every $n \in \mathbb N$ and by \Cref{bridges} the graph $T^i \cup T^j$ is an end faithful connected spanning subgraph of $G$.

For the construction of the spanning tree packings note that the graph $G_n' := G \cont{V_n \setminus V_{n-1}}$ is $2k$-edge connected. Let $u_n$ be the contracted vertex in $G_n'$ that corresponds to $V_{n-1}$ and denote by $v_n$ the only other contracted vertex in $G_n'$. Then the set $S_n$ of edges incident to $u_n$ is a $u_n$-$v_n$-cut of minimal cardinality. This implies that we can apply \Cref{treecut2} to each of the $G_n'$.

Start by selecting an arbitrary spanning tree packing $\mathcal T _1 = \set{T_1^1,T_1^2, \ldots , T_1^k}$ of $G_1$. Such a spanning tree packing exists by \Cref{treecut2}.

For $n \geq 2$ we will inductively define $\mathcal T_{n}$ from $\mathcal T_{n-1}$ by selecting a spanning tree packing $\mathcal T_n'$ of $G_n'-u_n$. Consequently \Cref{contract2} can be applied to find a spanning tree packing of $G_{n}$ for which \ref{twotrees1} holds. 

In order to take care of \ref{twotrees2} we will bridge one possible gap in every step of the construction process. For this purpose let $w_1,w_2,w_3,\ldots$ be an enumeration of the vertices of $G$ and choose a function
\[
	\varphi\colon \mathbb N \to \mathbb N^2 \times \set{(i,j) \mid 1 \leq i,j \leq k \text{ and }  i \neq j}
\]
such that $\varphi^{-1}(s)$ is infinite for every $s$. This function will be used to make sure that every persistent gap is bridged infinitely often in the following way: if $\varphi(n) = (n_1,n_2,i,j)$ and there is a gap $\Gamma$ in $T_{n-1}^i$ that separates $w_{n_1}$ and $w_{n_2}$ we construct a $\Gamma$-bridge in $T_{n}^j - V_{n-1}$. So if the gap $\Gamma$ is persistent there will be arbitrarily large sets of disjoint $\Gamma$-bridges since $\varphi^{-1}(n_1,n_2,i,j)$ is infinite.

All that is left to show now is that we can construct such a $\Gamma$-bridge, i.e., that we can choose the spanning tree packing of $G_n' - u_n$ accordingly. For this purpose let $P$ be the unique path in $T_{n-1}^i$ that connects $w_{n_1}$ and $w_{n_2}$. Note that we may assume that $w_{n_1}, w_{n_2} \in V_{n-1}$ and that $P$ uses the unique contracted vertex in $G_{n-1}$. Otherwise no gap $\Gamma$ separates $w_{n_1}$ and $w_{n_2}$ and thus there is nothing to show. 

Denote by $a_n$ and $b_n$ the vertices in $G_n'$ that are incident to edges used by $P$ such that $a_n,b_n \neq u_n$. Both $a_n$ and $b_n$ are contained in $W_n$ and since $G[W_n \setminus V_{n-1}]$ is connected, there is an $a_n$-$b_n$-bypass of $v_n$ in $G_n' - u_n$. So \Cref{treecut2} can be used to find a spanning tree packing $T_n' = \set{T_{n,1}', \ldots, T_{n,k}'}$ of $G_n' - u_n$ such that $T_{n,i}' \cup T_{n,j}'$ contains an $a_n$-$b_n$-bypass $P_n$ of $v_n$.

It is immediate that the union $P^* = P \cup P_n$ is a $w_{n_1}$-$w_{n_2}$-path in $T_n^i \cup T_n^j$ that does not use any contracted vertex. So $P^*$ contains a $\Gamma$-bridge in $T_n^j$. All edges of this bridge are contained in $T_n^j - V_{n-1}$ because $P^* \mid_{G\cont{V_{n-1}}}$ is a path in $T_{n-1}^i$.
\end{proof}

\begin{proof}[Proof Sketch of \Cref{onetree}]
Construct a sequence of spanning tree packings as we did in the proof of \Cref{twotrees}. 

This time however, instead of \Cref{treecut2} we will use \Cref{treecut3} to obtain $a_n$-$b_n$-paths that are  contained in $T_{n}^i$ and consequently the $v_{n_1}$-$v_{n_2}$-path $P^*$ is completely contained in $T_{n}^i$. So if at some point in the construction process there is a gap separating $v_{n_1}$ and $v_{n_2}$ it will eventually terminate. 

Now use \Cref{noinfgaps} to conclude that the limit of the spanning tree packings is an end faithful spanning tree packing of $G$.
\end{proof}

This completes the proof in the one-ended case. In the case of graphs with countably many ends we will again define a sequence $G_n$ of contractions of $G$ converging to $G$ and a sequence of spanning tree packings $\mathcal T_n$ of $G_n$ converging to the desired spanning tree packings. 

Since we would like to use \Cref{onetree,twotrees,treecut2} for the construction of the $\mathcal T_n$ we need to define the contractions in a way that this is possible. To motivate the definition given below consider for a moment the following situation: we have constructed a contraction $G_{n-1}$ and an end faithful spanning tree packing of this contraction. Now we wish to decontract a contracted vertex $v$ and use one of the theorems in the subgraph induced by the corresponding vertex set. Clearly it is essential that the cut formed by the edges incident to the contracted vertex is a cut of minimal cardinality in the minor induced by the set of vertices corresponding to $v$.

We have no sufficient means of controlling the order of the vertices being decontracted during the proof. Hence the above should be true for each contracted vertex at any time. This implies that we need cuts of minimal cardinality which do not interfere with each other.

\begin{defi}
Let $G = (V,E)$ be a graph, let $x \in V \cup \Omega$ and let $Y \subseteq (V \cup \Omega) \setminus \set{x}$. Let $S_y$ be a $x$-$y$-cut for $y \in Y$ and denote by $C_y$ the component of $G \setminus S_y$ in which $y$ lies. The set $\set{S_y \mid y \in Y}$ is said to be \emph{compatible} if for any $y,y' \in Y$ the cut $S_y$ does not contain an edge that connects two points in $C_{y'}$ , otherwise it is said to be \emph{incompatible}. If the set $\set{S_1,S_2}$ is compatible we will call the cuts $S_1$ and $S_2$ compatible.
\end{defi}
 
Clearly a compatible set of cuts of minimal cardinality is what is needed. The following definition can be used to provide a way of constructing such a set.

\begin{defi}
Let $G = (V,E)$ be a graph, $x \in V \cup \Omega$ and let $Y \subseteq (V \cup \Omega) \setminus \set{x}$. Define
\[
\mathfrak C _x ^Y = \set{S \mid S \text{ is a $x$-$y$-cut of minimal cardinality for some }y \in Y}
\]
and denote by $\mathfrak S_x^Y$ the power set of $\mathfrak C_x^Y$, i.e., the elements of $\mathfrak S_x^Y$ are sets of cuts.

Now we define a binary relation $\sqsubset$ on $\mathfrak S_x^Y \times \mathfrak S_x^Y$. We say that $\mathcal S \sqsubset \mathcal U$ if
\begin{enumerate}[label=(D\arabic*)]
\item \label{comp1} $\card{\mathcal S} \leq \card{\mathcal U}$,
\item \label{comp2} $S \subseteq \bigcup_{U \in \mathcal U} U$ for every $S \in \mathcal S$ and
\item \label{comp3} the component of $G \setminus \bigcup_{S \in \mathcal S} S$ in which $x$ lies is exactly the component of $G \setminus \bigcup_{U \in \mathcal U} U$ in which $x$ lies.
\end{enumerate}
\end{defi}

\begin{rem}
Clearly $\mathcal S \sqsubset \mathcal U$ and $\mathcal U \sqsubset \mathcal S$ implies that $\bigcup_{S \in \mathcal S} S = \bigcup_{U \in \mathcal U} U$. It is also an easily observed fact that the relation $\sqsubset$ is transitive and that whenever $U \in \mathfrak C _x ^Y$ and $U \notin \mathcal U$ then $\mathcal S \sqsubset \mathcal U$ implies that $\mathcal S \cup \set{U} \sqsubset \mathcal U \cup \set{U}$.
\end{rem}

Now starting with an arbitrary finite set $\mathcal U$ of minimal cuts we can find a compatible set $\mathcal S$ of minimal cuts which satisfies $\mathcal S \sqsubset \mathcal U$.

\begin{lem}
\label{compatible}
Let $G =(V,E)$ be a locally finite graph, $x \in V \cup \Omega$ and let $Y \subseteq (V \cup \Omega) \setminus \set x$. Let $\mathcal U \in \mathfrak S_x^Y$ be finite. Then there is a compatible set $\mathcal S \in \mathfrak S_x^Y$ such that $\mathcal S \sqsubset \mathcal U$.
\end{lem}

\begin{proof}
We will prove \Cref{compatible} by induction on $\card{\mathcal U}$. For $\card{\mathcal U} = 1$ there is nothing to show because a set of one cut is always compatible. 

For $\card{\mathcal U}>1$ let $U \in \mathcal U$ and apply the induction hypothesis to $\mathcal U \setminus \set{U}$ to obtain a compatible set $\mathcal S' \sqsubset \mathcal U \setminus \set{U}$ of cuts. By the above remark $\mathcal S' \cup \set{U} \sqsubset \mathcal U$. Clearly $\card{\mathcal S'} \leq \card{\mathcal U}-1$.

If the inequality is strict we can apply the induction hypothesis again to $\mathcal S' \cup \set{U}$ to obtain  $\mathcal S \sqsubset \mathcal S' \cup \set{U}$ where $\mathcal S$ is compatible. From the above remark it follows that $\mathcal S \sqsubset \mathcal U$ which completes the proof.

So assume that $\card{\mathcal S'} = \card{\mathcal U}-1$. Choose a cut $S_1 \in \mathfrak C_x^Y$ fulfilling $\mathcal S' \cup \set{S_1} \sqsubset \mathcal U$ with the property that the number of cuts $S' \in \mathcal S'$ that are incompatible with $S_1$ is minimal. Note that there is such a cut because $\mathcal S' \cup \set{U} \sqsubset \mathcal U$.

If $S_1$ is compatible with all cuts in $\mathcal S'$ we are done. So assume that there is a cut $S_2 \in \mathcal S'$ such that $S_1$ and $S_2$ are incompatible. For $i \in \set{1,2}$ let $y_i \in Y$ be such that $S_i$ is a $x$-$y_i$-cut of minimal cardinality.


For the next step of the proof we will need some definitions. Denote by $C_0$ the component of $G \setminus (S_1 \cup S_2)$ in which $x$ lies. Let $C_i$ be the component of $G \setminus S_i$ in which $y_i$ lies. Let $A_1$ be the set of edges connecting $C_1 \setminus C_2$ to $C_0$ and let $B_1$ be the set of edges connecting $C_1 \cap C_2$ to $C_2 \setminus C_1$. Analogously define $A_2$ and $B_2$. Let $C$ be the set of edges between $C_1 \setminus C_2$ and $C_2 \setminus C_1$ and let $D$ be the set of edges connecting $C_0$ and $C_1 \cap C_2$.

From the definitions it is clear that $A_1,A_2,B_1,B_2,C$ and $D$ are pairwise disjoint and that $S_i = A_i \cup B_i \cup C \cup D$. It is also clear that for the properties \ref{comp1} to \ref{comp3} it does not matter if an edge of $B_1$, $B_2$ or $C$ is contained in any cut.

We will now define a new cut which---depending on where $y_1$ and $y_2$ lie---can either be used to replace $S_1$ and $S_2$ so that we can apply the induction hypothesis again or contradicts the assumption that the number of cuts $S' \in \mathcal S'$ that are incompatible with $S_1$ is minimal.

\begin{itemize}[label=--]
\item Suppose that $y_1$ is contained in $C_1 \cap C_2$. Since $S_i$ is a $x$-$y_i$-cut of minimal cardinality it follows that
\begin{align*}
\card{A_1} + \card{B_1} + \card C + \card D &\leq \card{B_1}+\card{B_2}+\card D && \Rightarrow  & \card{A_1} + \card C &\leq \card{B_2},\\
\card{A_2} + \card{B_2} + \card C + \card D &\leq \card{A_1}+\card{A_2}+\card D && \Rightarrow  & \card{B_2} + \card C &\leq \card{A_1}.
\end{align*}
This implies that $\card C = 0$ and $\card{A_1} = \card{B_2}$. Hence
\[
\card{A_2} + \card{B_2} + \card C + \card D = \card{A_1} + \card{A_2} + \card D
\]
and thus the cut defined by $S^* := A_1 \cup A_2 \cup D$ is an $x$-$y_2$-cut of minimal cardinality. It is easy to see that $(\mathcal S' \setminus S_2) \cup  S^* \sqsubset \mathcal U$ and since $(\mathcal S' \setminus S_2) \cup  S^*$ has strictly less elements than $\mathcal U$ we can apply the induction hypothesis to find a compatible set $\mathcal S \sqsubset (\mathcal S' \setminus S_2) \cup  S^* \sqsubset \mathcal U$.
\item For $y_2 \in C_1 \cap C_2$ an analogous argument to the previous case works.
\item Finally assume that $y_1 \in C_1 \setminus C_2$ and $y_2 \in C_2 \setminus C_1$. Then
\begin{align*}
\card{A_1} + \card{B_1} + \card C + \card D &\leq \card{A_1}+\card{B_2}+\card C && \Rightarrow  & \card{B_1} + \card D &\leq \card{B_2}, \\
\card{A_2} + \card{B_2} + \card C + \card D &\leq \card{A_2}+\card{B_1}+\card C && \Rightarrow  & \card{B_2} + \card D &\leq \card{B_1}.
\end{align*}
This in particular implies that $\card D = 0$ and $\card {B_1} = \card {B_2}$ and thus
\[
	\card{A_1} + \card{B_1} + \card C + \card D = \card{A_1} + \card{B_2} + \card C.
\]
So the cut $S_1^* := A_1 \cup B_2 \cup C$ is an $x$-$y_1$-cut of minimal cardinality. 

Clearly $\mathcal S' \cup \set{S_1^*} \sqsubset \mathcal S' \cup \set{S_1} \sqsubset \mathcal U$.  Whenever $S_1$ and  $S' \in \mathcal S'$ are compatible it can easily be seen that $S_1^*$ and $S'$ are compatible as well (recall that $S'$ and $S_2$ are compatible). The cuts $S_1^*$ and $S_2$ are compatible while $S_1$ and $S_2$ are incompatible. So the number of cuts in $\mathcal S'$ that are incompatible to $S_1^*$ is strictly smaller than the number of cuts that are incompatible to $S_1$. This contradicts $S_1$ minimizing that number. \qedhere
\end{itemize}
\end{proof}

Another tool needed in the construction of the contractions are rays that completely cover an end in the following sense.

\begin{defi}
A ray $\gamma$ is said to \emph{devour} an end $\omega$ if every $\omega$-ray meets $\gamma$.
\end{defi}

From the existence of normal spanning trees (cf. \cite[Theorem 8.2.4]{1086.05001}) one can easily derive that such rays always exist.

\begin{lem}
\label{devour}
Let $G =(V,E)$ be a locally finite graph, $v \in V$, and let $\omega$ be an end of $G$. Then there is a ray $\gamma$ in $G$ that starts in $v$ and devours $\omega$.
\end{lem}

\begin{proof}
Consider a normal spanning tree with root $v$. The normal $\omega$-ray in this tree has the desired property.
\end{proof}

Now we have all means for constructing the sequence of contractions. The only thing that is still required for proving \Cref{onetreecountable,twotreescountable} are conditions to ensure that the topological spanning trees which we get in the limit have the desired properties. The next two results provide us with such conditions.

\begin{lem}
\label{oneray}
Let $G$ be a locally finite graph and let $T$ be a topological spanning tree of $G$ such that any two topological rays of $T$ that converge to the same end of $G$ have a common topological tail. Then $T \cap G$ is an end faithful spanning tree of $G$.
\end{lem}

\begin{proof}
Since $T$ contains no topological circle it doesn't contain a finite circle. So we only have to prove that $T \cap G$ is connected because in this case $T \cap G$ is a spanning tree that induces a topological spanning tree and thus it is end faithful by \Cref{treetoptree}.

So assume that $T \cap G$ was not connected and let $u$ and $v$ be vertices that lie in different components of $T \cap G$. Since $T$ is arcwise connected there is a $u$-$v$-arc $A$ in $T$ which clearly has to contain some end $\omega$. It is immediate that $A - \omega$ consists of two disjoint topological rays starting in $u$ and $v$ respectively. Both of those topological rays converge to the same end $\omega$ but they do not have a common topological tail, a contradiction.
\end{proof}

\begin{lem}
\label{subtree}
Let $G$ be a locally finite graph, let $H$ be a subgraph of $G$ and let $T \subseteq H$ be an end faithful spanning tree of $G$. Then $H$ is end faithful.
\end{lem}

\begin{proof}
Let $\omega$ be an end of $G$. Then there is a ray $\gamma$ in $T$ that belongs to $\omega$ since $T$ is end faithful. Hence $H$ contains a ray to every end of $G$. So we only need to show that every ray $\gamma'$ in $H$ that belongs to $\omega$ is equivalent to $\gamma$ in $H$. By the Star-Comb-\Cref{starcomb} the tree $T$ contains an infinite comb with all teeth in $\gamma'$ (note that $T$ is locally finite). The spine of this comb lies in $\omega$ and thus is either a tail of $\gamma$ or vice versa.
\end{proof}

Now we are finally in a position to prove \Cref{onetreecountable,twotreescountable}.

\begin{proof}[Proof of \Cref{onetreecountable}]
Let $(v_n)_{n\in \mathbb N}$ and $(\omega_n)_{n \in \mathbb N}$ be enumeration of all vertices and ends of $G = (V,E)$ respectively. We will now define contractions $G_n$ of $G$ such that $v_k \in G_n$ for all $k\leq n$ and the ends of $G_n$ are exactly $\set{\omega_1, \ldots ,\omega_n}$. Furthermore we require all of the $G_n$ to be locally finite. 

Then we construct end faithful spanning tree packings $\mathcal T_n = \set{T_n^1, \ldots , T_n^{k-1}}$ of $G_n$ such that $T_n^i\mid_{G_{n-1}} = T_{n-1}^i$. By \Cref{limtree} the limit of this sequence is a topological spanning tree packing of $G$ and we will apply \Cref{oneray} to show that it is indeed an end faithful spanning tree packing.

For the construction of the contractions $G_n$ let $G_0$ be the graph on one vertex, i.e., all of $V$ has been contracted. The graph $G_{n-1}$ is locally finite and does not contain an $\omega_n$-ray. Hence there must be a vertex $v$ in $G_{n-1}$ which resulted from contracting a subgraph $H_n$ of $G$ containing a tail of every $\omega_n$-ray.

Select a ray $\gamma$ in $H_n$ that devours $\omega_n$, i.e., $H_n - \gamma$ doesn't contain an $\omega_n$-ray. Such a ray exists by \Cref{devour}. For every $N>n$ such that a ray of $\omega_N$ lies in $H_n$ select a cut $S_N$ of minimal cardinality in $H_n$ separating $\omega_N$ from the set of vertices of $\gamma$. Note that every such cut is finite because there are only finitely many edge disjoint rays in $\omega_N$ that start in $\gamma$ (if there were infinitely many such rays then $\gamma$ would lie in $\omega_N \neq \omega_n$, a contradiction).

Now consider the component $C_{\omega_n}$ of $H_n - \bigcup S_N$ in which $\gamma$ lies. Clearly all rays in $C_{\omega_n}$ belong to $\omega_n$ because every other end has been separated from $\gamma$ by one of the cuts $S_N$. 

For each component $C$ of $H_n - C_{\omega_n}$ the set of edges connecting $C_{\omega_n}$ to $C$ is finite because otherwise we would by the Star-Comb-\Cref{starcomb} get an $\omega_n$-ray in $C$ contradicting the assumption that $\gamma$ devours $\omega_n$. Hence we can choose a finite set $\mathcal S_C^{\omega_n}$ of cuts $S_N$ containing all of these edges and by \Cref{compatible} we can find a compatible set of cuts $\tilde{\mathcal S}_C^{\omega_n} \sqsubset \mathcal S_C^{\omega_n}$.

Having chosen such compatible sets of cuts for each component of $H_n - C_{\omega_n}$ we contract every component of $H - \bigcup _C \bigcup _{S \in \tilde{\mathcal S}_C^{\omega_n}}S$ except $C_{\omega_n}$ to a vertex and denote the resulting graph by $H_n'$.

Define $G_n'$ to be the graph obtained from $G_{n-1}$ by replacing $v$ by $H_n'$. Clearly $G_n'$ is a contraction of $G$ with exactly the desired ends. If $v_n \in G_n'$ we choose $G_n = G_n'$. Otherwise decontract the contracted vertex $w$ of $G_n'$ that contains $v_n$ and denote the part of the resulting graph that corresponds to $w$ by $K_n$, i.e., $K_n$ is an induced subgraph of $G$ defined similar as $H_n$ above. 

For every end $\omega$ of $K_n$ choose a $\omega$-$v_n$-cut of minimal cardinality and denote the set of these cuts by $\mathcal S$. The component $C_{v_n}$ of $G - \bigcup_{S\in \mathcal S} S$ in which $v_n$ lies is obviously finite. Hence we can find a finite subset $\mathcal S^{v_n} \subseteq \mathcal S$ such that the component of $G - \bigcup_{S\in \mathcal S^{v_n}} S$ in which $v_n$ lies is exactly $C_{v_n}$. Now apply \Cref{compatible} to find a compatible set $\tilde{\mathcal S}^{v_n} \sqsubset S^{v_n}$.

Finally contract every component of $K_n -  \bigcup _{S \in \tilde{\mathcal S}^{v_n}} S$ except for the one containing $v_n$ to a vertex and denote the resulting graph by $K_n'$. We obtain $G_n$ from $G_n'$ by replacing $w$ by $K_n'$.

Clearly $G_n$ is a contraction of $G$. It is locally finite because 
\begin{itemize}[label=--]
\item $G_{n-1}$ is locally finite,
\item $G$ is locally finite, so all vertices of $H_n'$ and $K_n'$ that are not contracted have finite degree, and 
\item all newly contracted vertices have finite degree since the cuts that separate the corresponding vertex sets from the rest of the graph are finite.
\end{itemize}

The next step of the proof is to construct an end faithful spanning tree packing $\mathcal T_n$ of cardinality $k-1$ of $G_n$ out of such a packing $\mathcal T_{n-1}$ of $G_{n-1}$.  \Cref{contract2} and \Cref{treetoptree} imply that it is sufficient to find $k-1$ edge disjoint spanning trees of $K_n'$ and $k-1$ edge disjoint end faithful spanning trees of $H_n'$.

Consider the graph $H_n''$ obtained from $G_n$ by contracting $G_n - H_n'$ to a single vertex $x$ (so $H_n' = H_n'' -x$). By the construction of the sets $\tilde{\mathcal S}_C^{\omega_n}$ and $\tilde{\mathcal S}^{v_n}$ of cuts we know that the set $U$ of edges connecting $H_n'$ to $G_n - H_n'$ is a cut of minimal cardinality separating $G_n - H_n$ and some end $\omega$ of $H_n$. Hence $U$ is either an $\omega_n$-$x$-cut of minimal cardinality or an $x$-$y$-cut of minimal cardinality for some contracted vertex $y$ of $H_n'$. 

In the first case we can apply \Cref{onetree} directly to obtain an end faithful spanning tree packing of $H_n'$. 

In the latter case choose an $\omega_n$-$x$-cut $U'$ of minimal cardinality in $H_n''$. Clearly the graph $H_n'' - U'$ has two components one of which is infinite. Denote the vertex set of the finite component by $C_{\not{\infty}}$ and the vertex set of the infinite component by $C_\infty$.

Now $U$ is an $x$-$y$-cut of minimal cardinality for some contracted vertex $y$ of $H_n''\cont{C_{\not{\infty}}}$. Thus we can apply \Cref{treecut3} to find a spanning tree packing of cardinality $k-1$ of $H_n''\cont{C_{\not{\infty}}} - x$. Furthermore by \Cref{onetree} we can find $k-1$ edge disjoint end faithful spanning trees of $H_n''[C_\infty]$. \Cref{contract2} yields the desired end faithful spanning tree packing of $H_n''-x = H_n'$.

The spanning tree packing of $K_n'$ can be found by \Cref{treecut3} for the same reasons as above.

Finally we need to show that the limit of the sequence of spanning tree packings that we just constructed is indeed an end faithful spanning tree packing. By \Cref{limtree} it is a topological spanning tree packing because every end faithful spanning tree is also a topological spanning tree. Hence to be able to apply \Cref{oneray} we only need to show that for each $T \in \lim_{n \to \infty} \mathcal T_n$ and for every end $\omega_n$ of $G$ any two topological $\omega$-rays in $T$ have a common tail.

Assume that this was not the case and let $\gamma_1$ and $\gamma_2$ be topological $\omega_n$-rays with no common tail.  The restrictions $\gamma_1'$ and $\gamma_2'$ of the two rays to $G_n$ are again topological $\omega_n$-rays in $T_n = T \mid _{G_n} \in \mathcal T_n$ with no common tail. 

Since $G_n$ has only finitely many ends we may assume that $\gamma_1'$ and $\gamma_2'$ are rays (otherwise choose tails containing no ends). Now $\gamma_1'$ and $\gamma_2'$ are equivalent in $G_n$ but not in $T_n$, a contradiction to $T_n$ being an end faithful spanning tree of $G_n$.
\end{proof}

The construction process in the proof of \Cref{twotreescountable} is analogous to what was done in the previous proof hence we do not describe it in detail. Instead we indicate the modifications to be made and show how \Cref{subtree} can be used to show that the topological spanning tree packing has the desired properties.

\begin{proof}[Proof Sketch of \Cref{twotreescountable}]
For the construction of the the topological spanning tree packing follow the construction in the proof of \Cref{onetreecountable} (use \Cref{treecut2} instead of \Cref{treecut3} and \Cref{twotrees} instead of \Cref{onetree}).

To see that the union of two trees $T^i = \lim T_n^i$ and $T^j = \lim T_n^j$ is end faithful we construct a sequence of end faithful spanning trees of $T_n^i \cup T_n^j$ (and thus also of $G_n$) in every step of the construction process. We can do so by choosing an end faithful spanning tree of the graph $H_n'$ and a spanning tree of $K_n'$ which are contained in $T_n^i \cup T_n^j$ and applying \Cref{contract2}.

The limit of these spanning trees is an end faithful spanning tree by the same arguments as in the proof of \Cref{onetreecountable} and so we can apply \Cref{subtree} to show that $T^i \cup T^j$ is indeed an end faithful spanning subgraph of $G$.
\end{proof}

\section{An Application: Hamiltonian Cycles in Locally Finite Line Graphs}
\label{hamilton}

As already mentioned in the introduction the results of the previous section can be used to partially verify Georgakopoulos' \Cref{agelos}. More precisely we can show that the following special case holds.

\newtheorem*{restate:linegraph}{\Cref{linegraph}}
\begin{restate:linegraph}
The line graph of every locally finite $6$-edge connected graph with at most countably many ends has a Hamiltonian circle.
\end{restate:linegraph}

This theorem extends a result by \citet{brewfunk} who showed that Georgakopoulos' conjecture is true for $6$-edge connected graphs with finitely many ends all of which are thin.

To prove \Cref{linegraph} we will combine one of \Cref{onetreecountable,twotreescountable} with some known results from infinite graph theory. It is worth noting that the finite counterpart can be proved in a completely analogous way. Hence for every step we will also provide a brief sketch of the finite analogue. Basically the proof can be outlined as follows. 
\begin{enumerate}
\item Find two edge disjoint topological spanning trees $T, T'$ such that $T \cap G$ is an ordinary spanning tree of $G$.
\item Use the trees $T$ and $T'$ to construct an end faithful Eulerian subgraph $H$ of $G$.
\item From an Euler tour in $H$ construct a Hamilton cycle in $L(G)$.
\end{enumerate}

In the finite case one simply needs to find two edge disjoint spanning trees in the first step. So \Cref{tuttenashwilliams} can be used. Either of \Cref{onetreecountable,twotreescountable} can be used in the infinite case. Note that for a finite graph it suffices if it is $4$-edge connected while in the general case we require an edge connectivity of at least $6$. 

The second step uses the notion of fundamental circles where the fundamental circle of an edge $e$ with respect to a (topological) spanning tree $T$ is the unique (topological) circle in $T+e$. Now, in the finite case let $H$ be the subgraph of $G$ which contains an edge $e \in E$ if and only if it is contained in an odd number of fundamental circles of edges in $T'$ with respect to $T$. Clearly $H$ is connected because $T' \subseteq H$ and it is easily seen that all vertex degrees in $H$ are even. Hence $H$ is a spanning Eulerian subgraph of $G$.

In the infinite case we would like to use the same construction. To make this well-defined it is necessary that every edge is only contained in finitely many of the fundamental circles. Fortunately, by the following result this is always the case.

\begin{theo}[\citet{1050.05071}]
\label{sumwelldefined}
Let $G$ be a locally finite graph and let $T$ be a topological spanning tree of $G$. Then every edge is only contained in finitely many fundamental circles with respect to $T$.
\end{theo}

Next we need to define what we mean by an infinite Eulerian graph and to make sure that the construction used above really yields such a graph.
 
\begin{defi}
A \emph{topological Euler tour} of a graph $G$ is a continuous map $\sigma \colon S^1 \to  \overline G$ such that every inner point of an edge is the image of exactly one point of $S^1$. A locally finite graph is called \emph{Eulerian} if it admits a topological Euler tour.
\end{defi}

Above we used the well-known result from finite graph theory that a graph is Eulerian if and only if all vertex degrees are even. The following theorem by Diestel and Kühn provides a similar characterisation for infinite Eulerian graphs.

\begin{theo}[\citet{1063.05076}]
\label{eulerian}
For a locally finite graph $G$ the following two statements are equivalent.
\begin{enumerate}
\item $G$ is Eulerian.
\item Every cut of $G$ is either even or infinite.
\end{enumerate}
\end{theo}

With this result on hand it can easily be seen that the construction from the finite case also yields an Eulerian subgraph in the locally finite case.

\begin{prop}
\label{sumevencuts}
Let $G$ be a locally finite graph and let $T,T'$ be topological spanning trees of $G$. Let $H$ be the subgraph of $G$ which contains all edges of $T$ and those edges of $T'$ that are in an odd number of fundamental circles of edges in $T$. Then $H$ contains an even number of edges in $S$ for every finite cut $S$ in $G$.
\end{prop}

\begin{proof}
Denote by $K_S$ the set of fundamental circles that contain edges in $S$ for some finite cut $S$ and denote by $1_K$ the indicator function of the edge set of a given circle $K$. Clearly every such circle contains an even number of edges in $S$ and thus
\begin{align*}
\vert S \cap H \vert 
\equiv \sum_{e \in S}  \sum_{K \in K_S} 1_K(e) 
\equiv \sum_{K \in K_S} \sum_{e \in S} 1_K(e)
\equiv \sum_{K \in K_S} \vert K \cap S \vert 
\equiv 0 \mod 2 
\end{align*}
because all the sums are finite by \Cref{sumwelldefined}.
\end{proof}

The following result finishes the second proof step in the infinite case.

\begin{prop}
\label{treetoeuler}
Let $G$ be a locally finite graph and let $T$ and $T'$ be two edge disjoint topological spanning trees of $G$ such that $T \cap G$ is an ordinary spanning tree of $G$. Then $G$ has an end faithful spanning subgraph that admits a topological Euler tour.
\end{prop}

\begin{proof}
Let $H$ be the graph of \Cref{sumevencuts}. Then $H$ is connected because $T \subseteq H$ and it is end faithful by \Cref{subtree}.

Furthermore $H$ contains an even number of edges of every finite cut of $G$. Since $T$ is an end faithful spanning tree it contains an infinite number of edges of every infinite cut of $G$ and thus so does $H$. Hence every finite cut of $H$ is even and so $H$ is Eulerian. 
\end{proof}

In the third proof step for finite graphs we can use the following result by Harary and Nash-Williams and the fact that $H$ is Eulerian.

\begin{prop}[\citet{0136.44704}]
$L(G)$ is Hamiltonian if and only if there is a tour in $G$ which includes at least one endpoint of each edge.
\end{prop}

For locally finite graphs a similar result is known. 

\begin{defi}
A \emph{closed dominating trail} is a topological Euler tour of a connected subgraph $H$ of $G$ which contains at least one endvertex of each edge of $G$.
\end{defi}

\begin{prop}[\citet{brewfunk}]
\label{eulertohamilton}
Let $G=(V,E)$ be a locally finite graph. If $\overline G$ contains a closed dominating trail which is injective on the ends of $G$ then $L(G)$ is Hamiltonian.
\end{prop}

Note that the tour is required to be injective at ends. Even though the Eulerian subgraph that we constructed in the second proof step is end faithful a badly chosen Euler tour may visit an end more than once. However, the following result by Georgakopoulos can be used to find a topological Euler tour which is injective at the ends of $H$ and thus (since $H$ is end faithful) also at the ends of $G$.

\begin{theo}[\citet{pre05498547}]
\label{injectiveeuler}
If a locally finite multigraph has a topological Euler tour, then it also has one that is injective at ends.
\end{theo}

Since a topological Euler tour in a spanning connected subgraph clearly is a dominating closed trail this completes the proof of \Cref{linegraph}.

\section{The counterexample of Aharoni and Thomassen}

We mentioned earlier that \Cref{onetreecountable,twotreescountable} do not remain true if the graph has uncountably many ends. The following construction due to Aharoni and Thomassen \cite{MR982868} shows that for every $k \in \mathbb N$ there is a $k$-edge connected locally finite graph $G$ such that there is no edge disjoint pair consisting of a connected spanning subgraph of $G$ and an arcwise connected subspace of $\overline G$ containing all vertices of $G$. This clearly implies that the statements of \Cref{onetreecountable,twotreescountable} are not true for graphs with edge connectivity $\geq 6$ and uncountably many ends because both of the theorems give us such a pair.

For the construction let $k \in \mathbb N$ be fixed and choose a finite  $k$-edge connected graph $H$ such that there is a set $S$ of $k$ vertices whose pairwise distance is at least $k$.

Let $G_1 = H$ and let $V_0 = \emptyset$. In order to define $G_{n+1}$ start with $G_n$ and let $G_n' = G_n - V_{n-1}$. Subdivide every edge $e$ of $G_n'$ by adding $p_k(e)$ inner vertices where $p_k(e)$ be the number of paths of length $k$ in $G_n'$ containing $e$. We obtain a subdivision of $G_n$ whose vertex set is denoted $V_n$.

Now add a copy of $H$ for every path of length $k$ in $G_n'$ such that all the copies are mutually disjoint. Identify each vertex of a copy of $S$ with an inner point of a different edge in the corresponding path in $G_n'$. The identification is done in a way such that the copies of $H$ remain disjoint.

Define a limit graph $G$ on the vertex set $V = \bigcup_{n \geq 1} V_n$ by connecting a vertex $v \in V_n$ to all of its neighbours in $G_{n+1}$.

An inductive argument shows that $G_n$ and hence also $G$ is $k$-edge connected. It is also easy to see that the edge set of the subdivision of a path of length $k$ in $G_n'$ separates the corresponding copy of $H$ from $V_n$ and from al other copies of $H$ added in step $n+1$.

To see that there is no edge disjoint pair of a connected subgraph of $G$ and an arcwise connected subspace of $\overline G$ containing all vertices of $G$  consider a path $P$ connecting two vertices in the copy of $S$ in $V_1$. Every connected spanning subgraph of $G$ must contain such a path. There is a maximal $n$ such that this path intersects $V_n$. This implies that $P$ contains the subdivision of some path of length $k$ in $G_n'$ and thus a finite cut.

This shows that \Cref{onetreecountable,twotreescountable} cannot be extended to graphs with uncountably many ends. However, it might still be true to extend \Cref{linegraph} to such graphs.

\begin{prop}
\label{ahathom_hamiltonian}
Let $G$ be a graph constructed according to the above procedure. Then $L(G)$ is Hamiltonian.
\end{prop}

\begin{proof}
While it is not possible to find a suitable spanning tree packing we can still find a spanning end faithful Eulerian subgraph of $G$. Hence $L(G)$ is Hamiltonian by \Cref{injectiveeuler} and \Cref{eulertohamilton}.

Let $H$ be the $k$-edge connected graph used in the construction. Since $k \geq 4$ we know that $H$ contains a pair of edge disjoint spanning trees $T_1$ and $T_2$.  

It follows that $H$ contains a connected spanning Eulerian subgraph. Simply take all edges which are contained in an odd number of fundamental cycles with respect to $T_1$. It is easy to see that every vertex even degree in the resulting graph and it is connected because it contains all of $T_2$. Observe that flipping all edges---that is, remove an edge if it was present, add it if it was not---of  the unique $u$-$v$-path in $T_1$ gives a connected spanning subgraph of $H$ which contains a Eulerian $u$-$v$-trail, that is, the only vertices with odd degree in the subgraph are $u$ and $v$.

Now let $K$ be a connected spanning Eulerian subgraph of $H$ and let $K_{uv}$ be a connected spanning subgraph containing a Eulerian $u$-$v$-trail. We will use these graphs to iteratively define connected spanning Eulerian subgraphs $L_n$ of $G_n$.

Obviously we can choose  $L_1 = K$. Now we would like to construct $L_{n+1}$ out of $L_n$. Without loss of generality assume that $G_n - L_n$ does not contain any cycles. Otherwise take the graph obtained from $L_n$ by adding cycles as long as there is a cycle in $G_n$ whose edge set is contained in $G_n - L_n$. 

Let $L_{n+1}^0$ be the subdivision of $L_n$ contained in $G_{n+1}$, that is, $L_{n+1}^0$ is the graph obtained from $L_n$ by subdividing every edge $e$ of $G_n'$ by adding $p_k(e)$ inner vertices. The definitions of $G_n'$ and $p(e)$ are the same as in the construction of the graph $G_{n+1}$. Clearly, this is a connected subgraph of $G_{n+1}$ spanning all of $V_{n-1}$ in which all degrees are even.

So all we need to do is include the remaining vertices without making any degree odd or disconnecting the graph. We will now define a sequence $L_{n+1}^i$ of subgraphs of $G_{n+1}$ such that each $L_{n+1}^i$ is connected and has only vertices with even degree.

Denote by $\mathcal H$ the set of all copies of $H$ which have been added when we constructed $G_{n+1}$ out of $G_n$. For $H' \in \mathcal H$ let $S(H')$ be the corresponding copy of $S$. We call $H' \in \mathcal H$ \emph{pending}, if $H' \cap L_{n+1}^i = \emptyset$. In this case we also call vertices of $S(H')$ pending.

Let $F_{n+1}$ be the subdivision of  $G_n - L_n$ in $G_{n+1}$. Note that $F_{n+1}$ must be a forest because we assumed $G_n - L_n$ to be acyclic. Throughout the construction we will only make changes to $L_{n+1}^i$ in $F_{n+1}$ and in copies $H' \in \mathcal H$. In particular $L_{n+1}^0$ will be a subgraph of each $L_{n+1}^i$.

Call $H' \in \mathcal H$ \emph{unsettled} if every vertex of $H' \cap L_{n+1}^i$ lies on a path in $F_{n+1}$ connecting two pending vertices. In this case we also call vertices of $S(H')$ unsettled. Call a vertex \emph{settled} if it is not unsettled. Notice that every pending vertex is unsettled while the converse need not necessarily be true. 

Also notice that every vertex not contained in $F_{n+1}$ must be settled. This immediately implies that for any unsettled $H'$ and any two vertices in $S(H')$ there must be a path in $F_{n+1}$ connecting the two.

 Denote by $F_{n+1}^i$ the convex hull of the set of unsettled vertices in $F_{n+1}$, that is, $F_{n+1}^i$ is the subgraph of $F_{n+1}$ induced by all unsettled vertices and all vertices which lie on a path between two unsettled vertices in $F_{n+1}$.

To construct $L_{n+1}^{i+1}$ out of $L_{n+1}^i$ consider a leaf $v$ of $F_{n+1}^i$. Clearly, this is an unsettled vertex. 

First we claim that $v \notin L_{n+1}^i$. Indeed, if $v$ was contained in $L_{n+1}^i$ then it would have to lie on a path between two pending vertices. Since any pending vertex is unsettled this would be a path in $F_{n+1}^i$, hence $v$ would not be a leaf.

There is some $H' \in \mathcal H$ such that $v \in H'$. Choose any other vertex $u \in S(H')$ and let $P$ be the unique $u$-$v$-path in $F_{n+1}$. Then the graph $L_{n-1}^{i+1}$ is obtained from $L_{n-1}^i$ by flipping all edges along $P$ and adding a copy of the graph $K_{uv}$ in $H'$.

Clearly this makes $v$ a settled vertex. Since there is no change outside of $F_{n+1}^i \cup H'$ it follows that $F_{n+1}^{i+1}$ has strictly less vertices than $F_{n+1}^i$. Furthermore it is easy to see that every vertex in $L_{n+1}^{i+1}$ has even degree.

To show that $L_{n+1}^{i+1}$ is connected we need the following definition. A \emph{pivot vertex} is a vertex of $F_{n+1}$ which is incident to an edge in $L_{n+1}^i$ and an edge which is not contained in $L_{n+1}^i$. We now claim that every pivot vertex is connected to some vertex of $G_n$ by a path in $L_{n+1}^{i} - F_{n+1}^{i}$.

This is true for $L_{n+1}^0$ because the only pivot vertices were already contained in $G_n$. Inductively assume that it was true for $L_{n+1}^i$. Since no changes are made to $L_{n+1}^{i} - F_{n+1}^{i}$ it is clearly still true for every pivot vertex of $F_{n+1}^i$ after the modifications---$L_{n+1}^i-F_{n+1}^i \subseteq L_{n+1}^{i+1}-F_{n+1}^{i+1}$, so any path in the former graph is also a path in the latter graph. 

The only possible pivot vertices in $F_{n+1}^{i+1}$ which were not pivot vertices in $F_{n+1}^i$ are $u$ and $v$. By construction $u$ is connected to $v$ by a path outside of $F_{n+1}^{i+1}$. Now follow the $u$-$v$-path in $F_{n+1}$ starting at $v$. If we reach some vertex of $G_n$ before the first unsettled vertex we are done. Otherwise the first unsettled vertex must be a leaf in $F_{n+1}^{i+1}$. This implies that it is not contained in $L_{n+1}^{i+1}$, hence we must have passed at least one pivot vertex along the way. The path from $v$ to the first pivot vertex $v'$ that we passed is completely contained in $L_{n+1}^{i+1}-F_{n+1}^{i+1}$. It is an easy observation that all vertices on the $v$-$v'$-path are settled---an unsettled vertex on this path would imply an unsettled leaf of $F_{n+1}^{i+1}$ on it which cannot be contained in $L_{n+1}^{i+1}$. Since $v'$ is connected to some vertex of $V_n$ by a path in $L_{n+1}^{i+1} - F_{n+1}^{i+1}$ this completes the proof of the claim.

To see that $L_{n+1}^{i+1}$ is connected simply notice that every vertex of $L_{n+1}^{i+1}$ is connected by a path in $L_{n+1}^{i+1} - F_{n+1}^{i+1}$ to either a vertex of $G_n$ or a pivot vertex. The arguments are analogous to the arguments applied to the vertex $v$ above.

We mentioned earlier that the vertex sets of the graphs $F_{n+1}^i$ are strictly decreasing in $i$. Since $F_{n+1}^0$ is finite this implies that we end up with an empty vertex set after finitely many steps. In particular, there is some $i_0$ such that $L_{n+1}^{i_0}$ has no unsettled and hence also no pending vertices. So every $H' \in \mathcal H$ contains some vertex of $L_{n+1}^{i_0}$.

Now let $L_{n+1}$ be the graph obtained from $L_{n+1}^{i_0}$ by adding a copy of $K$ in every $H' \in \mathcal H$ whose vertex set is not contained in $L_{n+1}^{i_0}$.

From the sequence $L_n$ define a subgraph $L$ of $G$ by connecting every vertex in $V_n$ to all of its neighbours in $L_{n+1}$. Clearly $L$ contains an even number of edges in every finite cut of $G$. So we only need to show that $L$ is end faithful in order to complete the proof of the proposition.

Choose two rays $\gamma_1$ and $\gamma_2$ in $L$ belonging to the same end of $G$. Then $\gamma_1$ and $\gamma_2$ contain vertices $v_1$ and $v_2$ in the same component $C$ of $G - V_n$. If we choose $n$ large enough we may assume that those vertices both lie in $V_{n+1}$. By construction of $L$ there is a path in $L[V_{n+2} \setminus V_n] = L_{n+2} - V_n$ connecting the two vertices. Repeating this argument yields infinitely many vertex disjoint paths in $L$ connecting $\gamma_1$ to $\gamma_2$.

Hence any two rays which are equivalent in $G$ are als equivalent in $L$ and thus $L$ is end faithful.
\end{proof}

The construction in the proof of \Cref{ahathom_hamiltonian} is rather specific to the example of Aharoni and Thomassen. However, the general approach may be useful in a more general setting:
\begin{itemize}[label=--]
\item decompose the graph into finite parts,
\item choose a ``good'' spanning subgraph in each of the parts, that is, one where the right vertices have odd degree,
\item combine those subgraphs to obtain an end faithful spanning Eulerian subgraph,
\item use \Cref{injectiveeuler} and \Cref{eulertohamilton}.
\end{itemize}
In fact, it is easy to imagine a proof of \Cref{linegraph} which follows the same lines. \Cref{compatible} delivers the decomposition. Once we have chosen the spanning trees in the finite parts we can define a spanning subgraph in each of the parts accordingly and combine those subgraphs to obtain an end faithful Eulerian spanning subgraph of $G$.

This suggests that it may be possible to use the same idea for a wider class of graphs satisfying the conditions of Georgakopoulos' \Cref{agelos}. However, considering the amount of technical details needed in the proofs of \Cref{linegraph} and \Cref{ahathom_hamiltonian} it is likely to be difficult.

\section*{Acknowledgements}

I would like to thank Agelos Georgakopoulos for reading  earlier versions of this paper and for suggesting several improvements.

\section*{References}

\bibliographystyle{model1b-num-names}
\bibliography{bibliography}

\end{document}